\documentclass[]{amsart}
\usepackage{color}
\usepackage{cite}
\usepackage{amssymb}
\usepackage{amsfonts}
\usepackage{amsmath}
\usepackage{graphicx}
\usepackage{amsthm}
\usepackage{amsxtra}
\usepackage{subfigure}

\usepackage{hyperref}

\newtheorem{lemma}{Lemma}
\newtheorem{theorem}{Theorem}
\newtheorem{proposition}{Proposition}

{\theoremstyle{definition}}
{\theoremstyle{definition}}
{\theoremstyle{definition}\newtheorem*{remark}{Remark}}
{\theoremstyle{definition}}
\newtheorem{assumption}{Assumption}

\numberwithin{lemma}{section}
\numberwithin{proposition}{section}
\numberwithin{equation}{section}

\newcommand{\abs}[1]{\left\vert#1\right\vert}

\newcommand{\R}{\mathbf{R}}

\newcommand{\rmax}{r_{\max}}

\newcommand{\paren}[1]{\left(#1\right)}
\newcommand{\bracket}[1]{\left[#1\right]}
\newcommand{\set}[1]{\left\{#1\right\}}

\newcommand{\calN}{\mathcal{N}}

\newcommand{\N}{\mathbf{N}}

\newcommand{\norm}[1]{\left\Vert#1\right\Vert}
\renewcommand{\Re}{\mathrm{Re}}

\DeclareMathOperator{\rad}{rad}

\DeclareMathOperator{\bigo}{O}

\DeclareMathOperator{\arcsinh}{arcsinh}

\DeclareMathOperator{\mat}{Mat}
\DeclareMathOperator{\dir}{Dir}

\DeclareMathOperator{\rob}{Rob}
\newcommand{\inner}[2]{\left\langle #1,#2\right\rangle}

\title[Schr\"odinger-Poisson]{Nonlinear Bound States in a Sch\"odinger-Poisson System with External Potential}
\author{J.L. Marzuola}\thanks{ J.L.M. was support in part by NSF
  Applied Math Grant DMS-1312874 and NSF CAREER Grant DMS-1352353.  He
  wishes to thanks the Schr\"odinger Institute and the Mathematical
  Sciences Research Institute for graciously hosting him during part
  of this work.  He also thanks Hugh Bray for suggesting this problem
  in the first place as well as Richard Koll\'ar and Jianfeng Lu for
  helpful discussions.  The authors thank the anonymous referees for
  numerous helpful comments and generally improving the quality of the
  result with their suggestions.}
\address{University of North Carolina at Chapel Hill}
\email{marzuola@email.unc.edu}
\author{S.G. Raynor}\thanks{S. Raynor was supported by the Simons
  Foundation.  She would also like to thank the University of North
  Carolina at Chapel Hill for hosting her during part of this work.}
\address{Wake Forest University}
\email{raynorsg@wfu.edu}
\author{G. Simpson}\thanks{G. Simpson was supported by the US National
  Science Foundation grant DMS-1409018. }
\address{Drexel University}
\email{simpson@math.dexel.edu}

\begin{document}

\begin{abstract}
  We consider radial solutions to the Schr\"odinger-Poisson system in
  three dimensions with an external smooth potential with Coulomb-like
  decay.  Such a system can be viewed as a model for the interaction
  of dark matter with a bright matter background in the
  non-relativistic limit.  We find that there are infinitely many
  critical points of the Hamiltonian, subject to fixed mass, and that
  these bifurcate from solutions to the associated linear problem at
  zero mass.  As a result, each branch has a different topological
  character defined by the number of zeros of the radial states.  We
  construct numerical approximations to these nonlinear states along
  the first several branches.  The solution branches can be continued,
  numerically, to large mass values, where they become asymptotic,
  under a rescaling, to those of the Schr\"odinger-Poisson problem
  with no external potential.  Our numerical computations indicate
  that the ground state is orbitally stable, while the excited states
  are linearly unstable for sufficiently large mass.
\end{abstract}

\maketitle

\section{Introduction}
\label{s:intro}

We consider the existence and stability of stationary solutions to the
radial, focusing nonlinear Schr\"odinger-Poisson equation in $\R^3$
with focusing, Coulomb-like potential\footnote{By ``Coulomb like'', we
  mean that it decays like $|x|^{-1}$ as $r \to \infty$.},
\begin{equation}\label{SchPoi}
  i \partial_t \phi  - \Delta \phi +V(|x| ) \phi -{\mathcal{N}}(\phi) = 0.
\end{equation}
Under the ansatz $\phi(x,t) = e^{-i E t}u(x)$, the stationary
solution, $u$, satisfies a nonlinear elliptic equation with nonlocal
nonlinearity and long range potential function.  The time independent
problem takes the form:
\begin{equation}\label{NLS}
  - \Delta u +V(|x| ) u -{\mathcal{N}}(u) = -E u.
\end{equation}
Throughout, the external potential $V$ will be the solution of
\begin{equation}
  \label{pot}
  \Delta V = \rho (|x|)
\end{equation}
for $\rho > 0$, $\| \rho \|_{L^1} = Z$.  The techniques developed
here, analytically and numerically, can be modified to include the
case $\rho = Z \delta(x)$, corresponding to the classic Coulomb potential,
$-Z/|x|$.  Nonlinear Schr\"odinger equations, like \eqref{SchPoi},
with other external attractive, singular, potentials have also been
considered in quantum mechanical applications.  For instance, a $ -r^{-2}$ potential, together with a harmonic trapping potential
and a repulsive nonlinearity, was examined under spherical and cylindrical
symmetry in 
\cite{sakaguchi2011suppression1,sakaguchi2011suppression2}.  While similar methods to those
applied here could be implemented in such cases, due to our interest
in the application to general relativity we have focused on
$r^{-1}$-type potentials.  We will take an attractive nonlinearity
${\mathcal{N}}(u)$ to be of Schr\"odinger-Poisson type, given by
\begin{equation}
  \label{SchPoiNL}
  \calN (u) = ( |x|^{-1} * |u|^2 ) u = [(-\Delta)^{-1}|u|^2] u,
\end{equation}
where $( |x|^{-1} * |u|^2 )$ corresponds to the convolution of $|u|^2$
with the standard Green's function in $\R^3$.  This nonlinearity is
sub-critical with respect to the $L^2$ scaling and nonlocal.  In this
case, \eqref{NLS} is the Euler-Lagrange equation for the energy
functional
\begin{equation}\label{NLS-Lag}
  \mathcal{H} (u) = \int | \nabla u |^2 dx + \int V |u|^2 dx - \frac12  \int  \frac{ |u|^2 (x) |u|^2 (y)}{|x-y| } dx dy,
\end{equation}
subject to fixed mass
\begin{equation}\label{Mass}
  \mathcal{M} (u) =  \int | u|^2 dx,
\end{equation}
and $E$ plays the role of the Lagrange multiplier.

The nonlocal nonlinearity, \eqref{SchPoiNL}, arises in the
non-relativistic limit of an Einstein-Klein-Gordon system, which can
serve as a model for Dark Matter, \cite{Bray}. Following an idea of
Bray, this potential allows us to model the trapping of Dark Matter by
``bright matter.''  The potential is itself a solution to
$\Delta V = \rho$ for mass density $\rho$.  In a general relativistic
model proposed by Bray and others, stable excited states of the
Einstein-Klein-Gordon system including a background matter potential
representing the ``bright matter" have been sought, \cite{BG,BP}.
This is modeled by adding a mass density to the Einstein-Klein-Gordon
equations, which plays the role of the potential, $V$, in the
Schr\"odinger-Poisson model studied here.

Many of our results are applicable to other potentials and
nonlinearities, but we focus on Coulomb and Schr\"odinger-Poisson.
For instance, we might also study both super-critical and sub-critical
local nonlinearities, including the classical cubic nonlinearity,
\begin{equation}
  \label{cubicNL}
  \calN (u) = |u|^2 u,
\end{equation}
and a nonlinearity popular in density functional theory, representing
a Dirac exchange term, \begin{equation}
  \label{diracNL}
  \calN (u) = |u|^{\frac23} u.
\end{equation}
See \cite{anantharaman2009existence} for how such an attractive
nonlinearity arises in the LDA functional in density functional
theory.  In all cases we consider, the nonlinearities are assumed to
be focusing and the external potentials are assumed to be attractive.
The most significant differences amongst the cases will appear in the
large $E$ asymptotics.  Additional care in the analysis will also be
required for potentials which are not smooth, as in the pure Coulomb
potential for instance, along with non-smooth nonlinearities, such as
\eqref{diracNL}.

Here, we prove the existence of branches of radially symmetric
solutions to our system.  Each branch, as a function of the mass,
corresponds to solutions with a particular number of zero crossings in
the radial coordinate, and this number is invariant along the branch.
At mass zero, the branches terminate in the eigenstates associated
with the linear operator $-\Delta + V$.  Continuing the branches
requires a spectral assumption.  Specifically, we assume that
\begin{assumption}
  \label{spec_assumption}
  The kernel of the linearization of \eqref{NLS}, about a given
  solution, restricted to radial functions, is trivial.
\end{assumption}
\noindent In our numerical computations, we found that the discretized
operator did not have a kernel.

We are able to show that, at the very least, none of the branches
intersect.  In addition, we explore the high energy limit
($E \to \infty$), showing that these branches, should they continue
all the way to $E \to \infty$, connect to solutions of \eqref{NLS}
with $V = 0$.  We also examine the stability of the bound states, both
through a numerical examination of the spectrum, and through time
dependent simulations.  We find the ground state to be orbitally
stable, while the excited states, of sufficiently large $E$, are
linearly unstable.

Our work is organized as follows.  In Section \ref{s:linear}, we
review properties of the spectrum with Coulomb potentials and
establish the properties needed for a bifurcation analysis.  Next, in
Section \ref{s:existence}, we use a Lyapunov-Schmidt reduction to
construct a branch of bound states emanating from each linear
eigenvalue involving projection onto all the other discrete spectral
modes.  We then discuss how such branches behave as the nonlinear
eigenvalue $E \to \infty$, in Section \ref{sec:asymp}.  In Section
\ref{s:stability}, we review orbital stability and relate it to our
problem.  Then, in Section \ref{s:numerics}, we describe the numerical
methods we have used and present the results from various
time-dependent simulations and spectral stability calculations.  In
Section \ref{s:disc}, we discuss our calculations and simulations,
along with open problems.  Some additional bounds on unstable
eigenvalues are given in Appendix \ref{a:bds}

\section{Review of Linear Spectral Theory}
\label{s:linear}

In this section, we review some key results from linear spectral
theory for operators of the form of $H = -\Delta + V$.

\subsection{The Hydrogen Atom}
Recall that $V = -Z/|x|$ corresponds to the well known model of the
hydrogen atom, for which the eigenvalues and eigenfunctions are
entirely explicit; see \cite{gustafson2011mathematical}.  The
solutions to
\begin{equation}
  \label{eqn:Zcoulomb}
  - \Delta \psi_E - \frac{Z}{|x|} \psi_E=-E \psi_E
\end{equation}
can be obtained by power series methods, with eigenvalues
\begin{equation}
  \label{e:coulomb_evals}
  E = E_n\equiv  \frac{Z^2}{2n^2}, \ \ n \in \N,
\end{equation}
and corresponding radial eigenfunctions
\begin{equation}
  \label{e:coulomb_evecs}
  \psi_{n} (x) = e^{-\frac{Z |x|}{n}} P_{n-1} (\tfrac{Z |x|}{n}).
\end{equation}
Here, the $P_n (s) $ are the Laguerre polynomials $L_n^1 (s)$.  Each
$P_n$ has precisely $n$ positive zeros, hence $\psi_{n}$ has the
corresponding number of roots.

\subsection{Potentials with Coulomb like Decay at Infinity}

We will use variational methods to obtain the existence of infinitely
many radial excited states, with a sequence of eigenvalues approaching
zero from below.  For more on this type of analysis, see
\cite{RS4,Lenzmann}.
\begin{proposition}
  Assume that $V(x)$ is spherically symmetric and in $C^\infty$, and
  assume that $\exists Z \in \R^+$ such that
  \begin{equation}
    \label{e:coulomb_decay}
    \lim_{|x|\to \infty}|x|V(x)=-Z.
  \end{equation}
  Then $L = - \Delta + V(x)$, interpreted as a linear operator with
  form domain $H^1_{\rad}(\R^3)$, has an increasing infinite sequence
  of negative eigenvalues that approaches zero from below.
\end{proposition}

\begin{proof}
  Let $L_0 = -\Delta$. We apply the Rayleigh-Ritz technique, as in
  Section XIII.2 of Reed-Simon, to find the claimed infinite set of
  eigenvalues.

  Define
  \begin{equation}
    \label{e:mu}
    \mu_n(L) :=\sup_{\substack{S\subset H,\\ \mathrm{dim}(S)=n-1}}
    \inf_{\substack{\psi \in S^\perp,\\ \|\psi\|= 1}} \langle \psi, L\psi
    \rangle.
  \end{equation}
  Note that $L$ is bounded from below because $V$ is bounded and $L_0$
  is a nonnegative operator.  Therefore, it follows that if $P$ is a
  projection onto any $n$-dimensional subspace of $H$, then $\mu_n$ is
  bounded from above by the $n$-th eigenvalue of $PLP$ on $S$.
  (Theorem XIII.3 of \cite{RS4}).  \medskip Define the Rollnik class
  of potentials, $ \mathcal{R}$, by
  \begin{equation}
    \label{e:rollnik}
    \mathcal{R}\equiv \left\{V:\R^3\to\R \mid\displaystyle \int
      \frac{|V(x)||V(y)|}{|x-y|^2}dxdy<\infty\right\}.
  \end{equation}
  Also, define $(L^\infty)_\epsilon$ to be the set of functions with
  $L^\infty$ norm bounded by $\epsilon$.  As in Example $7$ on page
  $118$ of \cite{RS4}, if $V \in \mathcal{R}+(L^\infty)_\epsilon$,
  $\forall \epsilon > 0$, then $-\Delta + V$ is a form-compact
  perturbation of $L_0$ and therefore shares the same essential
  spectrum.  Coulomb potentials belong to the Rollnik class when cut
  off on any compact set.  Therefore, since the remainder is an
  arbitrarily small bounded perturbation, the Coulomb potential is in
  the class $\mathcal{R}+(L^\infty)_\epsilon$.  By standard Fourier
  analysis and a spectral perturbation argument, the essential
  spectrum of $L_0$ can be shown to be $[0, \infty)$.  In particular,
  zero is the bottom of the essential spectrum.

  By Theorem XIII.1 of \cite{RS4}, $\forall n \in \N$, $\mu_n$ is
  either the $n$-th eigenvalue of $L$ or $\mu_n=0$, the base of the
  essential spectrum.  Provided we can show that
  $\mu_n < 0 \ \forall n \in \N$, we can conclude that $L$ has
  infinitely many eigenvalues.  To do this, for each $n$ we will find
  appropriate $n$-dimensional spaces $H_n$ on which all eigenvalues
  are negative, and then apply the above described upper bound.

  This argument appears in the proof of Theorem XIII.6 of \cite{RS4}.
  Choose $\psi \in C_0^\infty(\R^3)$ satisfying $\psi \geq 0$ and
  supp$(\psi)\subset \{x: 1 < |x| < 2\}$, $\psi$ is radially symmetric
  and $\|\psi\|_{L^2} = 1$.  Define
  $\psi_R(x)=R^{-\frac32}\psi(\frac{x}{R}).$ Then
  supp$(\psi_R)\subset\{x:R<|x|<2R\}$ and $\psi_R$ satisfies the other
  conditions above.  For $R$ sufficiently large,
  \begin{align*}
    \langle \psi_R, L\psi_R \rangle & = \langle \psi_R, -\Delta\psi_R\rangle - \langle \psi_R, V(x)\psi_R \rangle \leq  \langle \psi_R, -\Delta\psi_R\rangle - \langle k\psi_R, {|x|^{-1}} \psi_R \rangle\\
                                    & \quad \leq R^{-2}\langle \psi, -\Delta\psi \rangle - kR^{-1} \langle \psi,\psi\rangle < 0.
  \end{align*}
  Fix $R_0$ sufficiently large so that this is true whenever $R > R_0$
  Now let $\phi_m=\psi_{2^mR_0}$ for $m=1\ldots n$.  It follows that,
  on $H_n=\mathrm{span}\{\phi_1,\ldots,\phi_n\}$, all eigenvalues of
  $PLP$ are negative, because these functions have disjoint support.
  Hence $L$ has an infinite sequence of negative eigenvalues
  approaching zero from below.
\end{proof}

\subsubsection{Sturm-Liouville Theory}
Let $(-E_n, \psi_n)$ be the eigenpairs for $L$ on $H$, ordered so that
$-E_n$ increases as $n$ increases.  We would like to know that, if
$-E_m>-E_n$, then $\psi_m$ has more zero crossings than $\psi_n$.
This requires a Sturm-Liouville-type argument on the radial equation
satisfied by the eigenfunctions. We first need a preliminary lemma
about the decay rate of our eigenfunctions:
\begin{lemma}
  \label{linearrdecay}
  For each $n$, $\psi_n$ has exponential decay as $r$ tends to
  infinity.
\end{lemma}

\begin{proof} 
  To see this, first let $\phi= r\psi_n$.  Then $\phi$ satisfies
$$
-\phi'' + V\phi=-E_n\phi.
$$

The proof follows directly from the asymptotic analysis of second
order linear systems in \cite{coddington1955theory}, Chapter $3$,
Theorem $8.6$.  In particular, this linear system can be seen to have
two solutions as $r \to \infty$, one of which is exponentially
decaying and one of which is exponentially growing and converging to
the rate of decay .  As a result, since an eigenfunction has decay, we
observe that for a sufficiently large $R > 0$ that
\[
  \lim_{r \to \infty} \phi (r) e^{ \int_R^r \sqrt{ E_n + V(s) } }ds =
  c
\]
for some constant $c$.
\end{proof}

Now, define $N_n$ to be the number of zeroes of $\psi_n$.  We will now
prove the following:
\begin{proposition}
  Under the same assumptions as above, whenever $-E_m > -E_n$,
  $N_{m}\geq N_{n}+1$.
\end{proposition}
\begin{proof}
  After the transformation in the proof of Lemma \ref{linearrdecay}
  this is a consequence of standard Sturm-Liouville theory in
  dimension one, see for instance the treatment in
  \cite{coddington1955theory}.
\end{proof}

Finally, we would like to confirm that each of these eigenvalues is
simple within $H^1_{\rad}$.
\begin{proposition}
  Each eigenvalue $E_n$ of the operator $H = -\Delta + V$,
  $\Delta V = \rho (|x|)$ with $\| \rho \|_{L^1} = Z$ in the class of
  radial functions, $H^1_{\rad}$, is simple.
\end{proposition}

\begin{proof}
  Again, this follows from Sturm-Liouville Theory once the
  transformation in the proof of Lemma \ref{linearrdecay} is used.
\end{proof}

\section{Existence of Nonlinear Bound States}
\label{s:existence}
We now prove the existence of nonlinear solutions bifurcating from
zero mass off of each discrete linear eigenvalue.  We will follow the
argument in Kirr-Kevrekidis-Schlizerman-Weinstein \cite{KKSW} to
obtain such bifurcation curves.  First let us construct the individual
bifurcation branches.

\begin{theorem}\label{bif}
  For a given $n \in \N$, let $(-E_n, \psi_n)$ be a simple eigenpair
  of $L:=-\Delta +V(x)$ in $H^1_{\rad}$, let $P$ be the projection
  onto the eigenspace, i.e.  $P u =\langle u, \psi \rangle \psi$, and
  let $Q=I-P$ be the spectral projection onto the rest of the spectrum
  of $L$. Define
  $\delta = \frac12\min\{|\mu - E_n|: \mu \in \sigma(L)\},$ and let
  $E$ be such that $0<E-E_n < \delta$.  Then, there exists a solution
  $u_E \in H^1_{\rad}$ to \eqref{NLS} with the same number of zero
  crossings as $\psi$.
\end{theorem}

\begin{proof}

  We seek nontrivial radial solutions $u$ to \eqref{NLS}, where
  $\|u\|_{L^2}$ is small, and therefore we expect that
  $u \sim c_0\psi_n$ and $E - E_n \sim 0 $ for $c_0$ also small.  For
  brevity, we write $\psi$ for $\psi_n$ in what follows.  In order to
  find $u$, we make the ansatz $u=c_0\psi+\eta$ with $Q\eta=\eta$.
  Substituting into \eqref{NLS}, we obtain
$$-\Delta(c_0\psi+\eta)+V(r)(c_0\psi+\eta)-\left ( \frac{1}{r} \ast |(c_0\psi+\eta)|^2 \right ) (c_0\psi+\eta)= E (c_0\psi+\eta).$$
Using the fact that $-\Delta \psi + V(r)\psi = -E_n \psi$ and that
$$Q(-\Delta +V(x)+E)\psi = Q(E-E_n)\psi =0,$$
we obtain
\begin{align*}
  c_0(E-E_n)\psi -P{\mathcal{N}}(c_0\psi+\eta) & = 0 \\
  (-\Delta+V(x)+E)\eta-Q{\mathcal{N}}(c_0\psi+\eta) & = 0
\end{align*}
where ${\mathcal{N}}(f)=(\frac{1}{r}*|f|^2)f$.  Note that
$\|{\mathcal{N}}(f)\|_{L^2} \leq k\|f\|_{H^2}^3$ for some $k > 0$, and
that ${\mathcal{N}}$ is a real analytic function in each argument.  We
will choose parameters $\nu, \rho$ later and we require that
$|c_0| < \nu$, $\|u\|_{H^2}<\rho$, and $|E-E_n| < \delta$.  Within
this open set, $(L-E_n)^{-1}Q$ is an analytic map from $L^2$ to $H^2$
with norm controlled by $\delta$, and hence it follows that
\[
  \|(L-E_n)^{-1}Q{\mathcal{N}}(c_0\psi + \eta)\|_H^2 \leq C(\delta)
  \|c_0\psi + \eta\|_{H^2}^3,
\]
and the map
\[
  F:=(c_0,E_n,\eta) \mapsto (L-E_n)^{-1}Q{\mathcal{N}}(c_0\psi + \eta)
\]
is real analytic.  Note that $F(0,E_n,0) = 0$ and $DF(0,E_n,0)=I$.
Hence, by the implicit function theorem there exist $\nu$ and $\rho$
so that on the open set described above, there is an analytic solution
$\eta(c_0,E_n)$ to
$\eta-(L-E_n)^{-1}Q{\mathcal{N}}(c_0\psi + \eta) =0$.  Note that
\[
  Q(L-E_n)^{-1}Q{\mathcal{N}}(c_0\psi + \eta)=Q(0)=0,
\] so
\[
  Q\eta = (H-E_n)^{-1}Q{\mathcal{N}}=\eta.
\]
So $\eta$ lies in the orthogonal projection away from $\psi$ as
desired.

Finally, by substituting back into the first equation, we obtain
$$c_0(E-E_n)\psi-P{\mathcal{N}}(c_0\psi_\eta(c_0,E_n))=0$$
with the condition $|c_0|^2+\|\eta(c_0,E_n)\|_{L^2}^2=\epsilon$ for
small fixed $\epsilon$.  Projecting onto $\psi$, we have that
$$E-E_n -|c_0|^2a - \frac{1}{c_0}\langle \psi, {\mathcal{N}}(c_0\psi+\eta)-{\mathcal{N}}(c_0\psi)\rangle = 0$$
where $a = \langle \psi, {\mathcal{N}}(\psi)\rangle$.  By the implicit
function theorem again, we obtain that there is a differentiable
function $h$ so that $E = h(c_0)$ in the allowed open interval.  We
may conclude that the desired solution $u_E$ exists for $E$ on this
curve.  Note that $h'(0) > 0$, so that $E > E_n$ in this regime.
\end{proof}
It follows that there is a bifurcation branch from each eigenvalue of
the linear problem.  As the spectral gap, measured by the number
$\delta$ in the above result, decreases, the range of $E$ for which
the theorem holds will be reduced.

We are also interested in the continuation of our branches away from
the zero mass limit, where we know they exist.  In particular, we
would like to know that they continue as $E \to +\infty$, and that the
branches do not intersect.

First, consider the matter of large values of $E$.  Define
$(\psi_j, E_j)$ to be the radial, normalized Coulomb eigenpairs of the
linear operator.  From Theorem \ref{bif} we have a $j$-th branch for
$E > E_j$, branching from $\psi_j$ at the zero mass limit. For each
branch, we follow \cite{KKP}, where smooth potentials are treated in
dimension one, and use the regularity of bound states with Coulomb
potentials from \cite{LS1977}.  Then, the Euler-Lagrange equations can
be seen as a map on $H^2$ functions given by
\begin{equation}
  \label{LS:EL-Hartree}
  F (Z,E;u) = -\Delta u + E u + V(|x|) u - (|x|^{-1} * |u|^2) u = 0.
\end{equation}
Consequently, away from mass zero (or for values of $E > E_j$), we can
apply the implicit function theorem directly to $F$ at $(E, u_E)$ to
construct a $C^1$ family of solutions $u_E \in H^2$ space under the
assumption that the linearization of the equation about solution
$u_E$,
\[
  L_+ = - \Delta + V(|x|) + E - (|x|^{-1} * |u_E|^2) - 2 (|x|^{-1} *
  (u_E \bullet ) )u_E,
\]
has no kernel.  This is Assumption \ref{spec_assumption} from the
introduction.  See the work \cite{Lenzmann} for a general treatment of
this problem with $V = 0$, where it is proven in their Proposition $2$
that for the ground state Hartree soliton, the kernel of $L_+$ is
trivial in the space of radial functions. We observe numerically below
(see Figure \ref{f:spec}) that each of our branches can be continued.
Using the same techniques as in Section \ref{s:numerics}, we found,
numerically, that the discretized $L_+$ operator lacks a kernel.

Moreover, the branches cannot cross.  Indeed, if two branches crossed,
then there would be a transition from a family of solutions with more
zeroes to one of fewer zeroes.  As a result, if this were to occur at
some point $r \geq 0$ along the curve, there would be a nonlinear
bound state with both value and derivative being $0$.  By ODE
uniqueness theory, this would be a trivial solution.\footnote{Note,
  there is a slight modification required at $r=0$ if $V = - Z/|x|$.
  In such a case, we must use instead of the normal radial condition
  $u_r (0) = 0$, the fact that we have $u_r (0) = -Z/2 u(0)$.}  Thus,
we conclude that if we were unable to continue a given branch in $E$,
it would not be due to branches crossing.

We note that, by the arguments for the proof of Theorem $1$ of
\cite{lions1980choquard}, for each $E>E_0$, there are an infinite
number of radial solutions with increasing energy.  As a result, we in
the next Section analytically and numerically consider the behavior of
solutions as $E\to \infty$, so long as our spectral assumption is met
and such branches can be continued.

\section{Limiting Behavior as $E \to \infty$}
\label{sec:asymp}

Following in the spirit of Section $4$ of \cite{KKP}, in this section
we consider the case $E \to \infty$ provided the lowest energy
solution branch can be uniquely continued.  The analytic results here
will apply to large $E$ behavior of the ground state branch, for the
generalized Coulomb-like equation
\begin{equation}
  \label{EL-HartreeGen}
  -\Delta u + E u + V(|x|) u - (|x|^{-1} * |u|^2) u = 0, \ \ \Delta V = \rho, \ \ \int \rho dx = Z,
\end{equation}
but with appropriate modifications a similar approach will apply to
\begin{equation}
  \label{EL-Hartree}
  -\Delta u + E u - \frac{Z}{|x|} u - (|x|^{-1} * |u|^2) u = 0.
\end{equation}
For the excited states, we lack a rigorous result on the kernel of the
linearized operator in the large $E$ limit.  Conditional on this
having trivial kernel, we can apply the same argument as in the case
of the ground state.



Without the external potential, the problem
\begin{equation}
  \label{e:NLS_V0}
  -\Delta \phi + E \phi - (|x|^{-1} * |\phi|^2)\phi = 0
\end{equation}
is solved by the $\phi_E = E \phi_1 ( \sqrt{E} x)$ where
\begin{equation}
  \label{e:NLS_V0_E1}
  -\Delta \phi_1 +  \phi_1 - (|x|^{-1} * |\phi_1|^2)\phi_1 = 0.
\end{equation}

Substituting the scaling $u = E \tilde u (\sqrt{E}x)$ into our
problem, we obtain
\begin{equation}
  \label{EL-HartreeGen-scaled}
  -\Delta \tilde u +  \tilde u + \frac{1}{ E}  V\left( \frac{x}{\sqrt{E}} \right)    \tilde u -(|x|^{-1} * | \tilde u |^2) \tilde u = 0.
\end{equation}
Note that in the pure Coulomb case we have
$V\left(x/{\sqrt{E}} \right) /E= -Z/(\sqrt{E}|x|)$, meaning that away
from zero, the potential is tending to zero, pointwise.  We wish to
show that $\tilde{u}$ is close to $\phi_1$ as $E \to \infty$ by using
the fact that the rescaled, smoothed Coulomb potential vanishes
pointwise.

Making the ansatz $\tilde u = \phi_1 + w$,
\begin{equation}
  \label{w:eqn}
  \begin{split}
    \tilde{L}w&=\mathcal{L}_+  w +  \frac{ 1}{ E}  V\left( \frac{x}{\sqrt{E}} \right)  w  \\
    &= -\frac{1}{ E}V\left( \frac{x}{\sqrt{E} } \right) \phi_1 +
    {\mathcal{N}} (\phi_1, w), \quad {\mathcal{N}} = O (\phi_1 w^2) +
    O (w^3),
  \end{split}
\end{equation}
where
\begin{equation}
  \label{Lplusdef}
  \mathcal{L}_+ w = -\Delta w + w - \left (\int \frac{ |\phi_1|^2 (y) }{|x-y|} dy \right ) w - 2 \left (\int \frac{\phi_1 w}{|x-y|} dy\right ) \phi_1.
\end{equation}
By results found in \cite[Proposition $2$, Appendix $A$]{Lenzmann},
there is a unique radial {ground state} $\phi_1$, which is positive
and exponentially decaying everywhere.  Furthermore, $\mathcal{L}_+$
is self-adjoint and non-degenerate with trivial kernel in the space of
radial functions.

For the class of bounded Coulomb potentials under consideration, since
$V \in L^\infty$, the multiplication operator
$ w \mapsto \frac{V(\frac{x}{\sqrt{E}})}{E} w$ is compact on $H^1$,
with norm at most $\frac{\|V\|_{L^\infty}}{E}$.  We also have that
$V \in L^2+(L^\infty)_\epsilon$, therefore, $\tilde{L}$ is a
relatively compact perturbation of the operator $\mathcal{L}_+$; see
Example $6$ on page $117$ of \cite{RS4}.

We have that $\tilde{L}$ is invertible on $H^1$ for sufficiently large
$E$ with an operator norm that is a perturbation of that of
$\mathcal{L}^+$.  Since $\tilde{L}^{-1}:H^1 \to H^1$ is bounded, we
have that
$\|\tilde{L}^{-1} \mathcal{N}(\phi_1, w)\|_{H^1}\leq C(
\norm{w}_{H^1}^2 + \norm{w}_{H^1}^3) $, while
\[
  \left \|\tilde{L}^{-1} \left(\tfrac{1}{ E}V\left( \tfrac{x}{\sqrt{E}
        } \right) \phi_1\right)\right\|_{H^1}\leq \tfrac{C}{E}.
\]
We therefore have that
\begin{equation*}
  \left \|\tilde{L}^{-1} \left(-\tfrac{1}{ E}V\left( \tfrac{x}{\sqrt{E} }
      \right) \phi_1 +  {\mathcal{N}} (\phi_1, w)\right)\right\|_{H^1} \leq C( E^{-1} + \norm{w}_{H^1}^2 + \norm{w}_{H^1}^3).
\end{equation*}
At the same time,
\begin{equation*}
  \begin{split}
    &\left \|\tilde{L}^{-1} \left({\mathcal{N}} (\phi_1, w_1 ) -
        {\mathcal{N}} (\phi_1, w_2 )\right)\right\|_{H^1} \\
    &\leq C(\norm{w_1}_{H^1} + \norm{w_1}_{H^1}^2 + \norm{w_2}_{H^1} +
    \norm{w_2}_{H^1}^2)\|w_1 - w_2\|_{H^1}.
  \end{split}
\end{equation*}
Now assume that $w$ is small, for instance, $\bigo(E^{-1/2})$.  Then,
by a standard iteration argument, a solution $w$ will be found in
$B_R(0)\subset H^1$, with $R \sim E^{-1/2}$ as $E$ becomes large.  To
conclude, we can construct solutions along our ground state branch
such that $w \to 0$ as $E \to \infty$ and the profile of our solutions
approaches $\phi_1$ for large $E$ as claimed. See Figure
\ref{f:logplots} for numerical exploration of this scaling limit,
which confirm the desired scaling as $E$ becomes large.

As a side note, one could prove that solutions to
\eqref{EL-HartreeGen} for a given $E$ have an $L^2$ norm indicated by
the scalings used above, then as in \cite{KKP} for the $1d$ case with
smooth potential, concentration compactness tools could be used to
give a simpler proof of the convergence explored for the ground state
branch.  However, as the nonlinearity has such nice algebraic
properties, we have taken the approach of using elliptic estimates
directly.

\begin{figure}[h]
  \includegraphics[width=6.25cm]{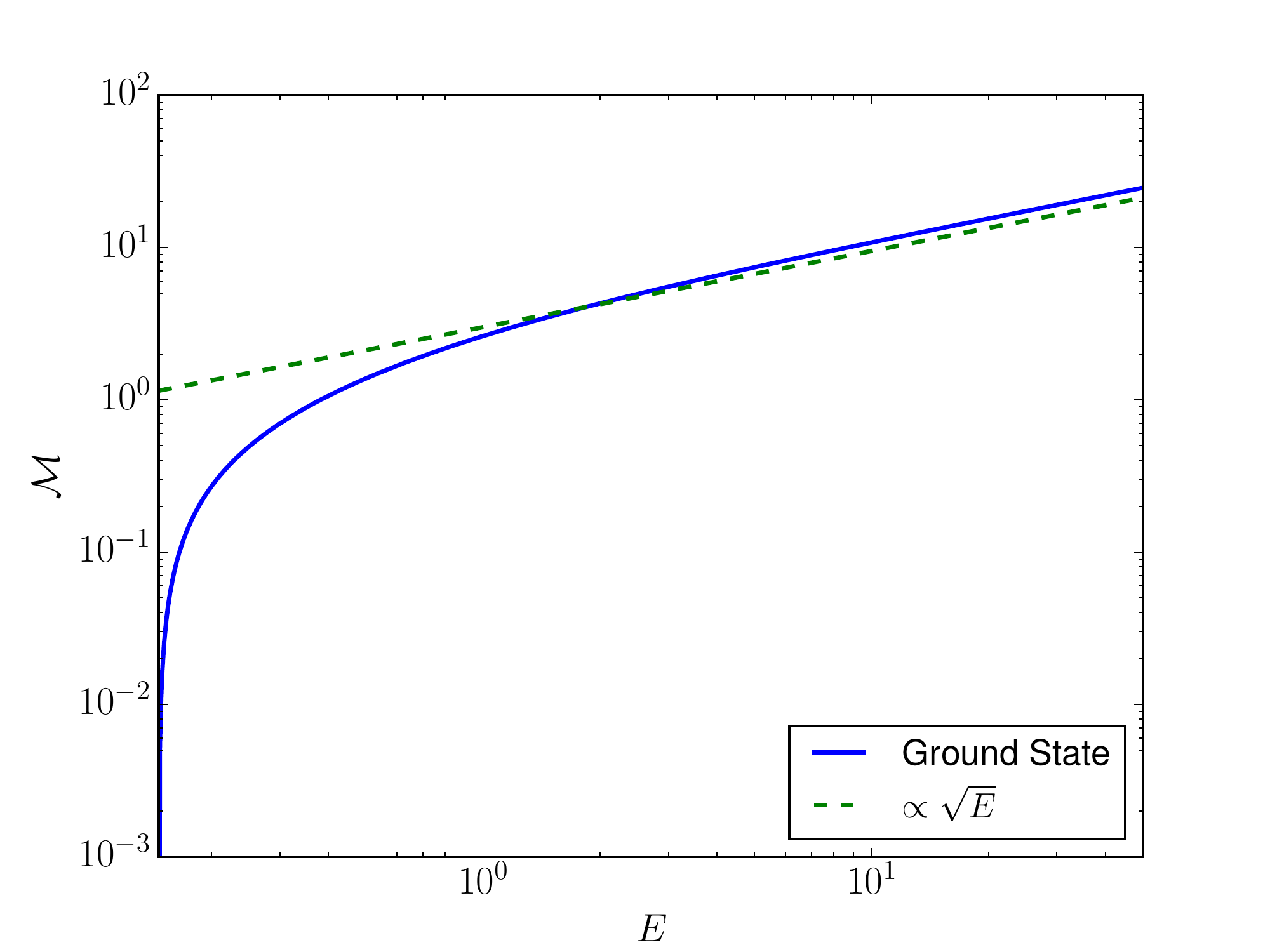}
  \includegraphics[width=6.25cm]{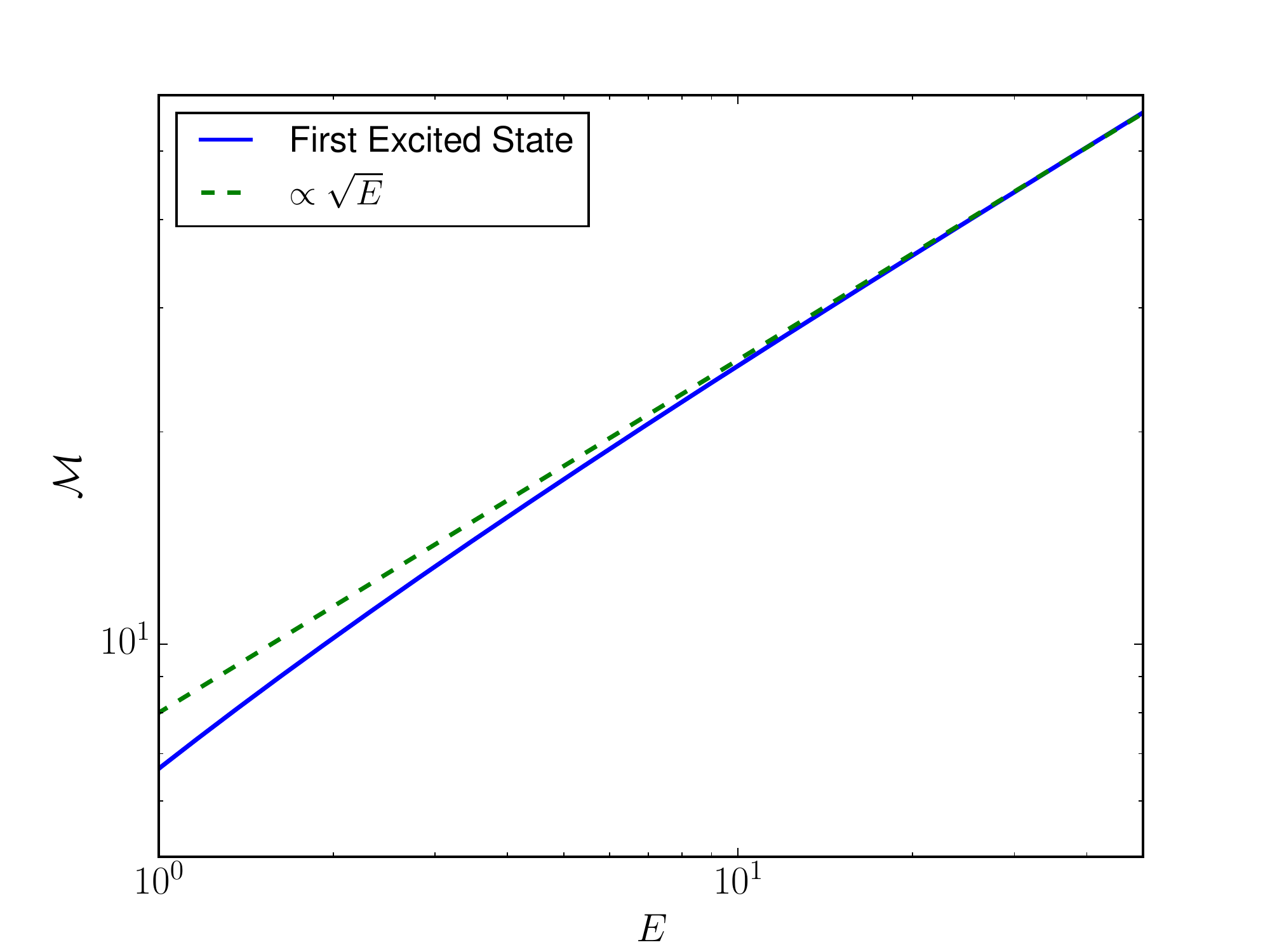}
  \caption{Plots of $E$ vs. mass for large $E$ for the first branch
    (left) and a zoom in on for large $E$ on the second branch
    (right).  Note, the slope approaches $\tfrac{1}{2}$ on the
    $\log - \log$ scale.}
  \label{f:logplots}
\end{figure}

\begin{remark}
  A dividend of this examination of the large $E$ limit is that it
  provides a strategy for computing excited state solutions for
  \eqref{e:NLS_V0}.  While, in practice, one would not solve for
  $E=\infty$, one could solve Schr\"odinger-Poisson for large values
  of $E$, to get, after rescaling, an approximation of the solution to
  the $V=0$ problem.  This preconditioned guess could then be fed to a
  Newton solver.  See \cite{Olson:2014ix} for a discussion on
  computing excited states along with an alternative strategy for
  obtaining states with a given number of zero crossings.


\end{remark}

\section{Stability}
\label{s:stability}

In the context of the stability works of Grillakis-Shatah-Strauss
\cite{GSS}, also known as the Vakhitov-Kolokolov criterion, we
consider the problem of orbital stability, restricted to radial
functions, of our solution.  By orbital stability, we mean that for
any $\epsilon>0$, there exists a $\delta$ such that if
$\norm{u_0 - \phi}_{H^1} \leq \delta$, then for all $t\geq 0$,
\begin{equation*}
  \inf_{\theta} \norm{u(t) - e^{i \theta} \phi }\leq \epsilon.
\end{equation*}
Orbital stability makes no claim as to any particular asymptotic
behavior.

To proceed, recall that we can write the linearized evolution
operator, in terms of real and imaginary parts, as
\begin{equation}
  \label{e:lin_op}
  H = J L = \begin{pmatrix} 0 & 1 \\ -1 & 0\end{pmatrix}\begin{pmatrix}L_+ & 0\\ 0 & L_-\end{pmatrix} = \begin{pmatrix}  0 & L_-\\ -L_+ & 0 \end{pmatrix},
\end{equation}
where
\begin{subequations}
  \label{e:lin_ops}
  \begin{align}
    L_- &= -\Delta + E + V(x) - (|x|^{-1} |u_E|^2), {\mbox{ and }}\\
    L_+ &= -\Delta + E + V(x) - (|x|^{-1} \ast|u_E|^2) - 2( |x|^{-1} \ast (u_E \bullet)) u_E.
  \end{align}
\end{subequations}
Also, define the scalar function
\begin{equation}
  \label{e:scalar_energy}
  d(E) = \mathcal{H}(u_E) + E \mathcal{M}(u_E).
\end{equation}
Recall then, the results of \cite{GSS} (see, also,
\cite{Weinstein,Weinstein85,Grillakis,Grillakis:1990p116,KM}), adapted
to this problem.  Let $p(d''(E)) = 1$ if $d''(E)>0$ and let
$p(d''(E))=0$ otherwise.  Let $n(L) = n(L_-) + n(L_+)$ be the number
of negative eigenvalues of the operators.  Subject to some assumptions
on well-posedness of the flow, the existence of the bound states, and
the ability to decompose the spectrum of $L$, we have:
\begin{theorem}[from \cite{Grillakis:1990p116}]
  Assume $d''(E) \neq 0$, then
  \begin{description}
  \item[Stability] If $n(L) = p(d'')$, the bound state is orbitally
    stable,
  \item[Instability] If $n(L) - p(d'')$ is odd, then the soliton is
    orbitally unstable.
  \end{description}
\end{theorem}
In some important cases, such as NLS with a power nonlinearity, these
properties can be deduced analytically; this is the content of some of
the formative works on soliton stability.  For our problem, however,
we must numerically compute the bound state of energy $E$, compute
$d''(E)$, and then count the number of eigenvalues of the discretized
operators $L_\pm$.

These computational tasks, detailed below, are readily addressed.
Briefly, we find that the ground state soliton is orbitally stable, as
is typical for subcritical problems.  For the excited states, we find
that $n(L) - p(d'')$ is even in the cases we compute; this case is not
addressed by the above theory.  We thus perform both direct
computation of the spectrum $JL$, as well as time dependent
simulations of the excited states with finite perturbations.  For
sufficienlty large $E$, the excited states appear to be linearly
unstable.


To simplify these computations, slightly, we recall from
\cite{Grillakis,Weinstein} that, for nonlinear bound states in one
parameter, an important identity can be obtained for $d''(E)$.
Observe that, in general,
\[
  d'(E) = \left( \frac{\delta \mathcal{H}}{\delta u}, \partial_E
    u_E\right) + \mathcal{M}(u_E) + E\left( \frac{\delta
      \mathcal{M}}{\delta u }, \partial_Eu_E\right).
\]
Since the variations are evaluated at $u_E$, and $u_E$ satisfies the
PDE,
\begin{equation}
  \label{e:slope_condition}
  d'(E) = \mathcal{M}(\phi_E).
\end{equation}
Thus, $d''(E)>0$ if and only if $\mathcal{M}(u_E)$ is a decreasing
function in $E$ corresponding to the Vakhitov-Kolokolov criterion.  In
our computations, we find that, in all cases examined, $d''(E) >0$;
$\mathcal{M}(u_E)$ is an increasing function of $E$.

\section{Numerical Computation of Bound States and Stability}

\label{s:numerics}

Our approach to computing the nonlinear bound states to \eqref{NLS} is
to start with a bound state with the desired number of zero crossings
for the associated linear problem
\begin{equation}
  \label{e:NLS_lin}
  -\Delta u + V(|x|) u = -E u,\quad \norm{u}_{L^2} = 1.
\end{equation} 
We then perform numerical continuation to obtain the desired nonlinear
bound state.  During the continuation, the number of zero crossings is
invariant.

While it is convenient to think of the linear bound state as the zero
mass limit of the nonlinear bound state, this is impractical for
numerical continuation.  Instead, we augment \eqref{NLS} with the
artificial continuation parameter, $\gamma\in [0,1]$, to become
\begin{equation}
  \label{e:continuation}
  -\Delta u +V(|x|) u - \gamma \mathcal{N}(u) = - E u, \quad
\end{equation}
Then, along a sequence of $\gamma$ values,
\begin{equation*}
  0 = \gamma_0 < \gamma_1 <\ldots < \gamma_{n_\gamma -1} = 1,
\end{equation*}
$(u^{(i)}, E^{(i)})$ pairs are computed, all with $L^2$ norm of unity.

Once the value at $\gamma=1$ is obtained, the mass constraint is
relaxed, and $E$ is varied to determine, for instance, $d''(E)$.  At
each value of $E$, the eigenvalues of matrix discretizations of
$L_\pm$ are computed.

For concreteness, $V$ is the smooth radial function solving
\begin{equation}
  \label{e:concreteV}
  \Delta V = \frac{1}{2}e^{-r}, \quad  V(r) = \frac{1}{2}e^{-r} -\frac{1}{r}(1-e^{-r}).
\end{equation}

\subsection{Computation of the Linear States}

To begin with, we compute the eigenvalues of \eqref{e:NLS_lin} using
its associated weak form and piecewise linear, radial finite elements.
A Neumann condition is applied at the origin, and a ``big box''
homogeneous Dirichlet approximation is made at $r_{\max}$, assumed to
be sufficiently large.  For $\lambda>0$, the states will be
exponentially localized, so this is a reasonable approximation.
However, since the point spectra tend to zero, the decay rates at the
$n$th eigenvalue of the order,
$\propto r^{ n } \exp(-\sqrt{\lambda}r)$ will demand ever larger
values of $\rmax$ in order to be well approximated.  For this reason,
we will only consider the first few eigenstates.

For a fixed $\rmax$, the corresponding linear system is
\begin{equation}
  \label{e:linear_eig}
  \underbrace{{\bf K}_{\rm Dir.}}_{\text{Stiffness Matrix}} + \underbrace{{\bf
      V}_{\rm Dir.}}_{\text{Potential Matrix}} = -\lambda\underbrace{ {\bf
      M}_{\rm Dir.}}_{\text{Mass Matrix}}.
\end{equation}
Recall that since our basis is the set of hat functions,
$\{\varphi_i\}$, on $[0, r_{\max})$,
\begin{gather*}
  ({{\bf K}_{\rm Dir.}})_{ij} = \int_0^{r_{\max}} \varphi_i'(r)
  \varphi_j'(r) r^2dr , \quad ({{\bf V}_{\rm Dir.}})_{ij}=
  \int_0^{r_{\max}} V(r)\varphi_i(r)
  \varphi_j(r) r^2dr, \\
  \quad ({{\bf M}_{\rm Dir.}})_{ij}= \int_0^{r_{\max}} \varphi_i(r)
  \varphi_j(r) r^2dr.
\end{gather*}
${\bf V}_{\rm Dir.}$ is computed using numerical quadrature.  The
eigenstates are then computed, as illustrated in Figure
\ref{f:linear_profiles}.  Since the states are highly localized, we
use a nonuniform mesh, given by
\begin{equation}
  \label{e:nonumesh}
  r_j = \sinh(\xi_j), \quad \xi_j = j\delta \xi, \quad
  \delta\xi=\frac{\arcsinh{\rmax}}{N}, \quad j =0,\ldots, n.
\end{equation}
This spaces the nodes linearly near the origin and exponentially
further apart as $j$ increases.

\begin{figure}
  \includegraphics[width=8cm]{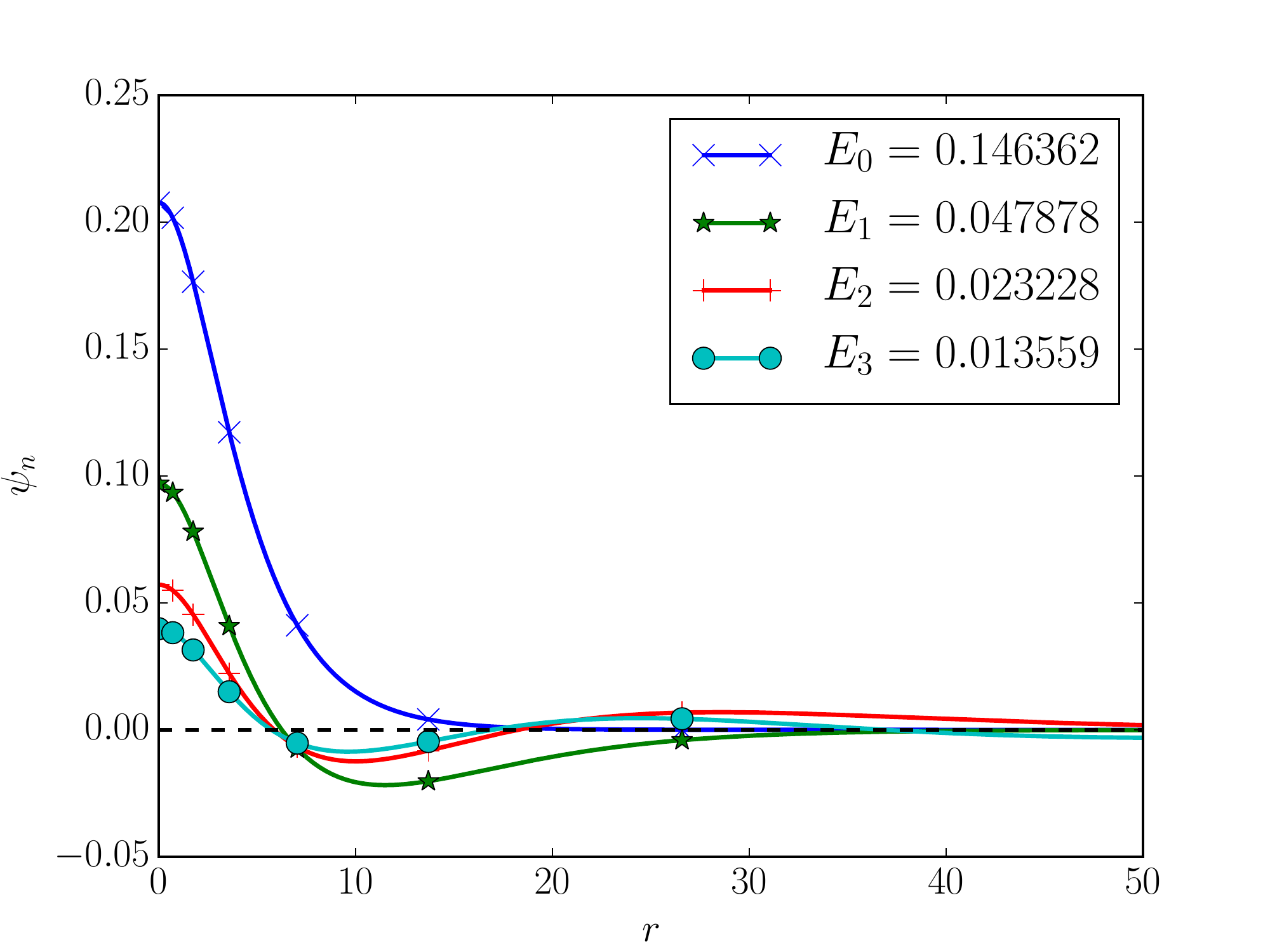}
  \caption{The ground state and the first few excited states of the
    associated linear problem, \eqref{e:NLS_lin}, with $V(r)$ given by
    \eqref{e:concreteV}.  Computed using \eqref{e:linear_eig} on the
    mesh given by \eqref{e:nonumesh} with $n = 4000$ and
    $\rmax = 100$.}
  \label{f:linear_profiles}
\end{figure}

\subsection{Computation of the Nonlinear States}

Given the solution to the linear problem at $\gamma=0$, we must now
use a nonlinear solver to obtain the desired solution at $\gamma=1$.
This is performed using the Python implementation of
\cite{Shampine:2006vr} to solve \eqref{e:continuation}.  This software
is available at \url{https://pythonhosted.org/scikits.bvp_solver/}.

\subsubsection{First Order System}
To use this software package, we must first reformulate our problem as
a first order system, with associated boundary conditions.  We first
remove the nonlocality, by writing our problem as a system of
constrained second order equations:
\begin{subequations}
  \label{e:schod_poi_system1}
  \begin{gather}
    E u - u'' - \frac{2}{r}u' +  V(r) u - w u  = 0,\\
    - w'' - \frac{2}{r}w' = {|u|}^2, \\
    \int_0^{\infty} |u|^2 r^2 dr = 1
  \end{gather}
\end{subequations}
This is then transformed into the aforementioned first order system,
with $v=u'$, $z=w'$, and $m(r)$ being the accumulated mass in $[0,r]$.
\begin{equation}
  \label{e:first_order}
  \frac{d}{dr}\begin{pmatrix}
    u\\v\\w\\z\\m
  \end{pmatrix} = \frac{1}{r}\begin{bmatrix}0 & 0 & 0 & 0 & 0\\
    0 & -2 & 0 & 0 &0\\
    0 & 0 & 0 & 0 & 0\\
    0 & 0 & 0 & -2 & 0\\
    0 & 0 & 0 & 0 & 0
  \end{bmatrix}\begin{pmatrix}
    u\\v\\w\\z\\m
  \end{pmatrix} + \begin{pmatrix} v\\ V(r) u-\gamma w u + E u \\ z \\
    -|{u}|^2 \\ u^2r^2 \end{pmatrix}
\end{equation}
In the above expressions, we will use $V(r)$ as given by
\eqref{e:concreteV}.

\subsubsection{Boundary Conditions}
It is now necessary to specify boundary conditions for
\eqref{e:first_order}.  First, we have the natural boundary conditions
that $u$ and $w$ be radially symmetric functions.  Furthermore, the
mass density, $m(r)$, must be zero at the origin.  This yields the
following three boundary conditions:
\begin{equation}
  \label{e:natural_bc1}
  u'(0) = v(0) = 0, \quad w'(0) = z(0) = 0, \quad m(0) = 0.
\end{equation}
Next, since the computation is performed on a large, but finite,
domain, suitable approximate boundary must be imposed at $r=r_{\max}$.
First, we observe that, since $u$ is localized, we can enforce the
fixed mass condition by approximating
\begin{equation}
  \label{e:mass_farfield_bc}
  m(\rmax) = \int_0^{\rmax}|u|^2 r^2 dr= 1.
\end{equation}
Next, we first write the equation for $w$ as
\begin{equation*}
  (w' r^2)' = -r^2 |u|^2\Rightarrow w'(r) = -\frac{m(r)}{r^2}.
\end{equation*}
Since, for large $r$, $m(r)$ is approximately constant, we have
\begin{equation*}
  w(r) \approx \frac{m(\rmax)}{r}.
\end{equation*}
This gives rise to our next approximate boundary condition,
\begin{equation}
  \label{e:w_farfield_bc}
  z(\rmax) + \frac{1}{\rmax} w(\rmax) = 0.
\end{equation}
Finally, at large $r$,
\begin{equation}
  \label{e:u_farfield}
  0\approx  Eu-u'' - \frac{2}{r}u' + V u  - w u\approx Eu-u'' - \frac{2}{r}u' -\frac{1}{r} u  - \frac{m(\rmax)}{r} u
\end{equation}
and we arrive at the approximate Robin condition
\begin{equation}
  \label{e:u_farfield_bc}
  v(\rmax) + \paren{\frac{1}{\rmax} + \sqrt{E}- \frac{1 + m(\rmax)}{2
      \rmax\sqrt{E}}}u(\rmax) = 0
\end{equation}

The reader may ask, why, in \eqref{e:w_farfield_bc} and
\eqref{e:u_farfield_bc}, we have not replaced $m(\rmax)$ by one, as in
\eqref{e:mass_farfield_bc}.  The reason is that, in the first stage of
our computation, we will solve for $E$, as an unknown, at fixed
$L^2$-mass.  Subsequently, we will allow $E$ to be a specified
parameter, and $L^2$ will be an unknown.  When $E$ is specified, we
discard \eqref{e:mass_farfield_bc}, but continue to use
\eqref{e:w_farfield_bc} and \eqref{e:u_farfield_bc}, with $m(\rmax)$
an unknown that is solved for.

\subsubsection{Fixed Mass Profiles}

\begin{figure}

  \subfigure{\includegraphics[width=6.25cm]{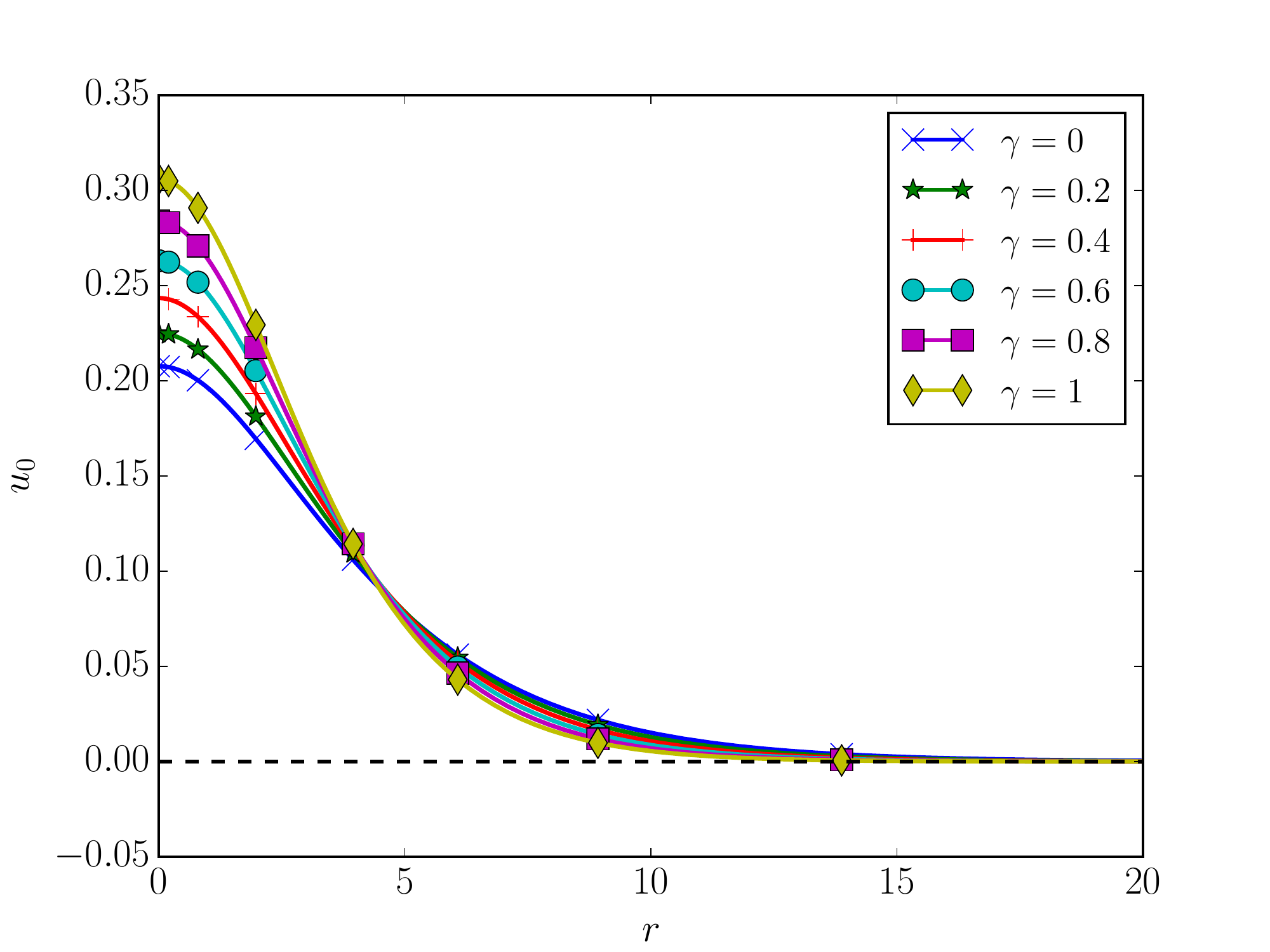}}
  \subfigure{\includegraphics[width=6.25cm]{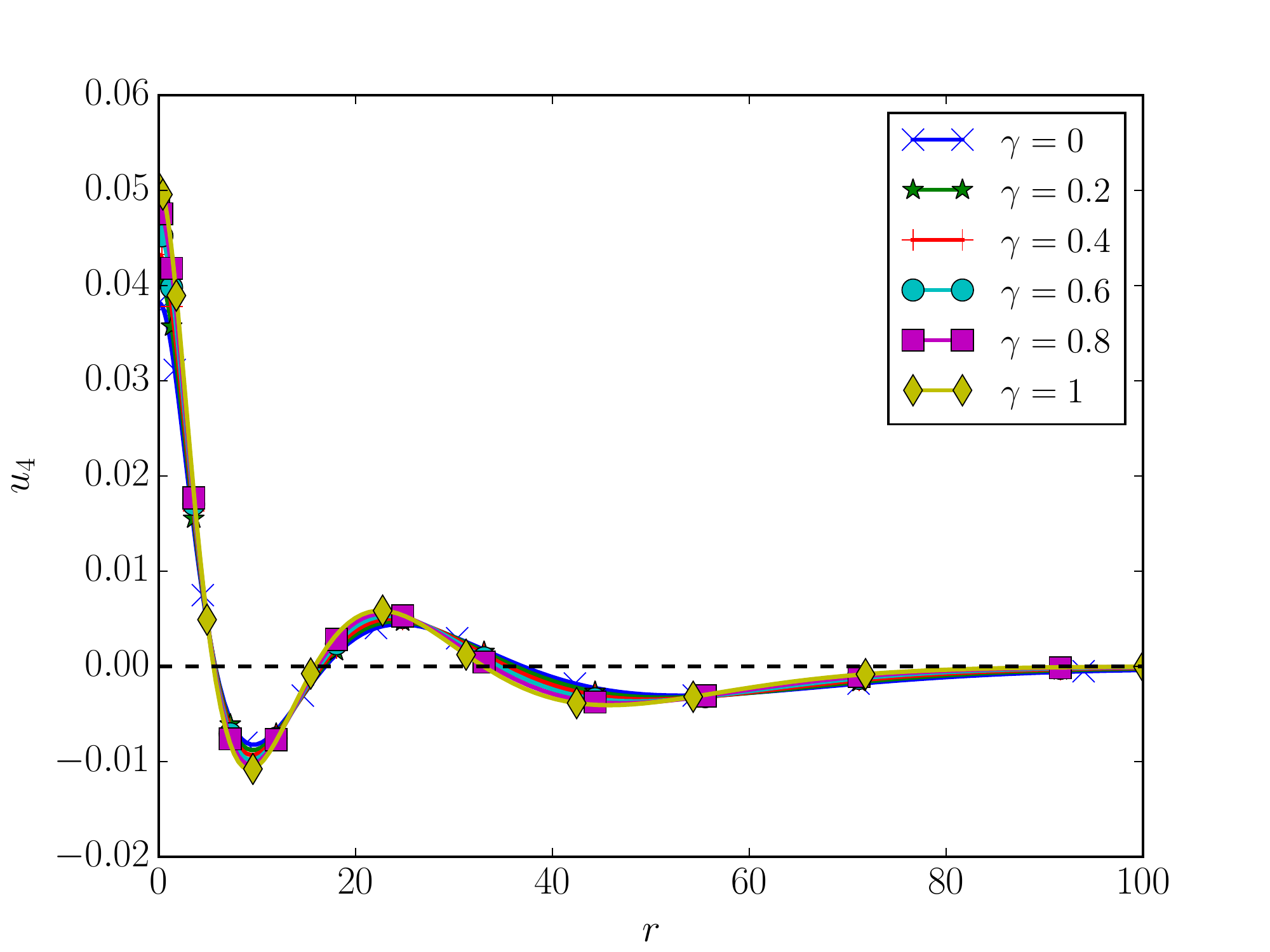}}

  \caption{Some of the solutions computed during continuation of
    $\gamma$ from zero to one in \eqref{e:continuation}.  In all
    computed cases, the number of zero crossings was invariant during
    the continuation.}
  \label{f:cont_profiles}
\end{figure}

For mass fixed at one, our continuation strategy produces the sequence
of solutions indicated in Figure \ref{f:cont_profiles} for the ground
state and an excited state with four zero crossings. Note that if
$u^\gamma$ solves \eqref{e:continuation}, with mass one, then
$U^\gamma = \sqrt{\gamma} u^\gamma$ solves
\[
  -\Delta U^\gamma + V U^\gamma - \mathcal{N}(U^\gamma) = - E
  U^\gamma, \quad \mathcal{M}[U^\gamma] = \gamma.
\]
Thus, this figure can also be interpreted as the branching of $E$ off
of the linear eigenvalues from the linear zero amplitude solutions.
Several profiles at $\gamma =1$ are shown in Figure \ref{f:gamma1}.




\begin{figure}
  \includegraphics[width=8cm]{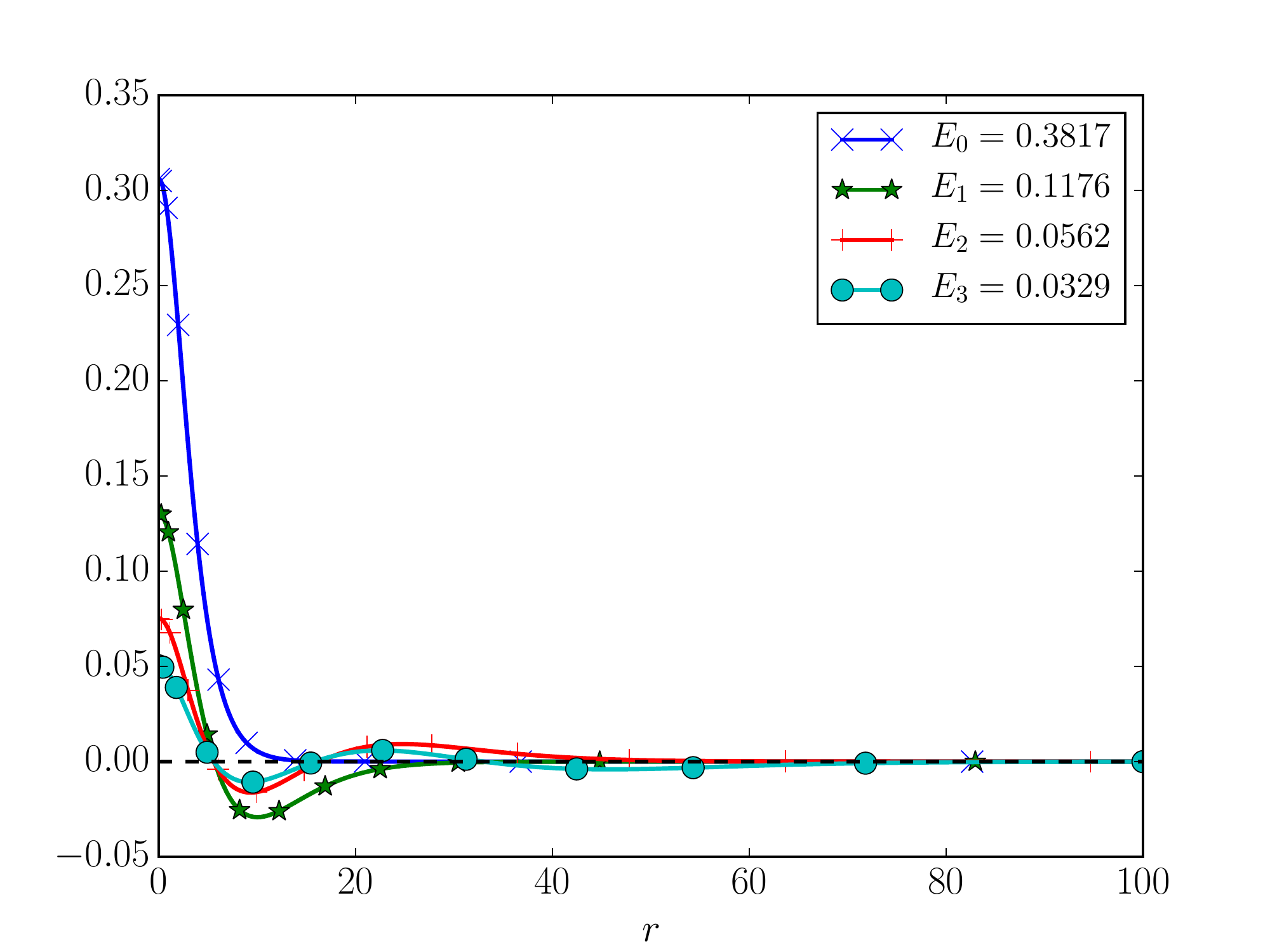}

  \caption{Profiles for the ground state and several excited states at
    $\gamma=1$.  All have mass one.}

  \label{f:gamma1}
\end{figure}

\subsubsection{Variable $E$ Profiles}

Starting from the mass one profiles, we vary $E$ about the value
computed above, and compute a collection of profiles for each of the
nonlinear bound states.  We plot the mass as a function of $E$ in
Figure \ref{f:massEcurves}.  Recall from the slope condition,
\eqref{e:slope_condition}, that since these appear to be strictly
increasing in all cases, $p(d'')=1$ in all of the cases we have
computed.  We speculate that this is true for all cases of this
problem. The maximum value of $E$ at which we computed is, for each
branch, twice the value of $E$ corresponding to the mass one problem.
Each branch terminates at the corresponding eigenvalue of the
associated linear problem.

\begin{figure}
  \includegraphics[width=8cm]{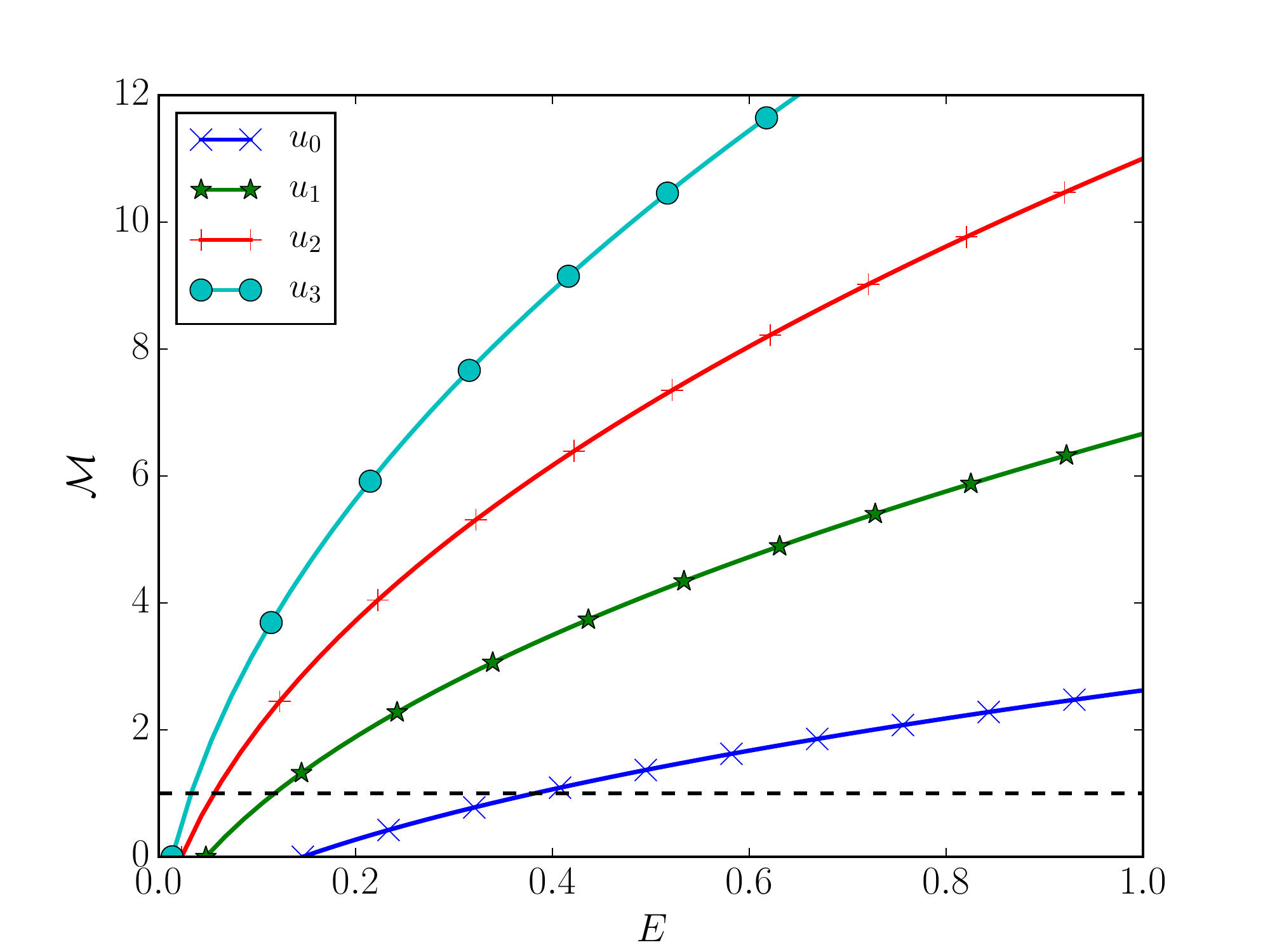}
  \caption{Mass as a function of $E$ for several branches of the
    problem.  In all cases, the curves appear to be monotonically
    increasing.}
  \label{f:massEcurves}
\end{figure}


\subsubsection{Remarks}

In our experience, this approach was highly robust.  The continuation
strategy from the linear problem to the nonlinear problem required a
modest number of intermediate values of $\gamma$; $\delta \gamma=0.05$
was used in the above calculations.  A slight difficulty occurs when
considering states which branch from linear states with eigenvalues
close to the origin.  As mentioned earlier, while these will decay
exponentially, the successively slower decay will require larger and
larger domains.

\subsection{Stability Calculations}

To proceed with an analysis of the stability, we first need to
discretize the operators, $L_\pm$, and then compute their eigenvalues.

\subsubsection{Discretization of the Operators}
To compute the spectrum of $L_\pm$, we continue to work within the FEM
context.  The one subtlety to this is how to represent the nonlocal
linear operator, $(|x|^{-1} \ast(u \bullet)) u$, in the weak form.
Let $T$ denote this operator.  We approximate it as follows.  First
note that the Galerkin weak form is
\begin{equation}
  \inner{T \varphi_i}{\varphi_j} = \langle u \underbrace{(|x|^{-1} \ast(u \varphi_i))}_{\equiv\psi^{(i)}} ,\varphi_j\rangle.
\end{equation}

Observe that $\psi^{(i)}$ solves
\begin{equation}
  \label{e:psi1}
  -\Delta \psi^{(i)} = u \varphi_i, \quad \partial_{r}\psi^{(i)}(0)=0,
  \quad \lim_{r\to\infty}\psi^{(i)}(r) =0.
\end{equation}
An artificial boundary condition is now needed to numerically solve
this on our computational domain. First, we observe that since
$\varphi_i$ have finite support and $u$ is highly localized, at large
values of $r$, $-\varphi^{(i)}\approx 0$.  Thus, we introduce the
Robin condition at $\rmax$:
\begin{equation}
  \label{e:psi2}
  \partial_{r}\psi^{(i)}(\rmax)   + \tfrac{1}{\rmax}\psi^{(i)}(\rmax)
  = 0.  
\end{equation}
The $\psi^{(i)}$ are then approximated in the space
${\rm span}(\varphi_0, \ldots, \varphi_n)$ by solving
\begin{equation}
  K_{\rob}\boldsymbol{\psi}^{(i)} = U_{\mat} \mathbf{e}_i,
\end{equation}
for $i =1,\ldots, n$.  The matrix $U_{\mat}$ is given by
\begin{equation}
  \label{e:Umat}
  (U_{\mat})_{ij} = \int u \varphi_i \varphi_j.
\end{equation}
Note that $K_{\rob}$ is an $n+1\times n+1$ matrix and
$\boldsymbol{\psi}^{(i)} \in \R^{n+1}$.  Here, we take $U_{\mat}$ to
be $n+1\times n+1$, with $i,j=0,\ldots, n$.  Taking
$\mathbf{e}_i\in \R^n$,
\begin{equation}
  \boldsymbol{\psi}^{(i)}  = K_{\rob}^{-1} U_{\mat} I^{n+1,n} \mathbf{e}_i,
\end{equation}
where $I^{n+1,n} $ is an $n+1 \times n$ matrix with ones along the
main diagonal and zeros elsewhere.  Thus,
\begin{equation}
  \begin{split}
    \inner{T \varphi_i}{\varphi_j} =\langle{u
      \psi^{(i)}}{\varphi_j}\rangle&\approx \sum_{k=0}^n
    \inner{u\varphi_k}{\varphi_j}(K_{\rob}^{-1} U_{\mat} I^{n+1,n}
    \mathbf{e}_i)_{k}\\
    &\quad = (U_{\mat} K_{\rob}^{-1} U_{\mat} I^{n+1,n}
    \mathbf{e}_i)_j
  \end{split}
\end{equation}
Mapping back into the set of elements vanishing at $\rmax$, the weak
form of $T$ corresponds to the matrix
\begin{equation}
  \label{e:Tmat}
  T_{\mat} = I^{n,n+1} U_{\mat} K_{\rob}^{-1} U_{\mat} I^{n+1,n}
\end{equation}

Thus, the Galerkin FEM forms of the eigenvalue problems for $L_\pm$
are
\begin{subequations}
  \label{e:Lops}
  \begin{align}
    L_-: &&    (K_{\rm Dir} + E M_{\rm Dir} + U_{\mat} ){\bf v}  &= -\mu M_{\rm Dir}
                                                                   {\bf v},  \\
    L_+: &&    (K_{\rm Dir} + E M_{\rm Dir} + U_{\mat} - 2 T_{\mat}){\bf v}  &= -\mu M_{\rm Dir}
                                                                               {\bf v}.
  \end{align}
\end{subequations}
While it is intimidating to contend with the nonlocal operator, which,
in discretized form, induces a dense matrix, we found that this was
readily handled by SciPy, \cite{scipy}.

\subsubsection{Eigenvalues of $L_\pm$}

In Figure \ref{f:spec}, we plot the numerically computed negative
spectrum for the linearized operators.  We note here that as described
above for the full problem, the continuous spectrum for our linearized
operators starts at $E > 0$ and that infinitely many positive
eigenvalues of the operators exist in between $0$ and $E$ due to the
slow decay of the external potential.  We note that under the
assumptions that $L_+$ is invertible along the branch and using the
nodal count for $L_-$ from the linear solutions, one can show that the
number of negative eigenvalues for $L_+$ and $L_-$ does not change
from that in the case of the linear problem, which is once again
verified here numerically.  We also see that the topological structure
of the modes increases in a very similar fashion to that of the model
Hydrogen atom problem.  In the computed cases, for $E$ in excess of
the zero mass limit, $L_+$ of $u_j$ has $j+1$ negative eigenvalues,
while $L_-$ has $j$ eigenvalues.  Thus, we always obtained $2j +1$
negative eigenvalues.  This implies that the ground state is orbitally
stable, since $n(L) = p(d'')=1$.  However, it is inconclusive for the
excited states, since the difference between $n(L)$ and $p(d'')=1$ is
always a nonzero even number.  These were computed using the mesh
\eqref{e:nonumesh}, with $\rmax=100$ and $n=2000$.

\begin{figure}
  \subfigure{\includegraphics[width=6.25cm]{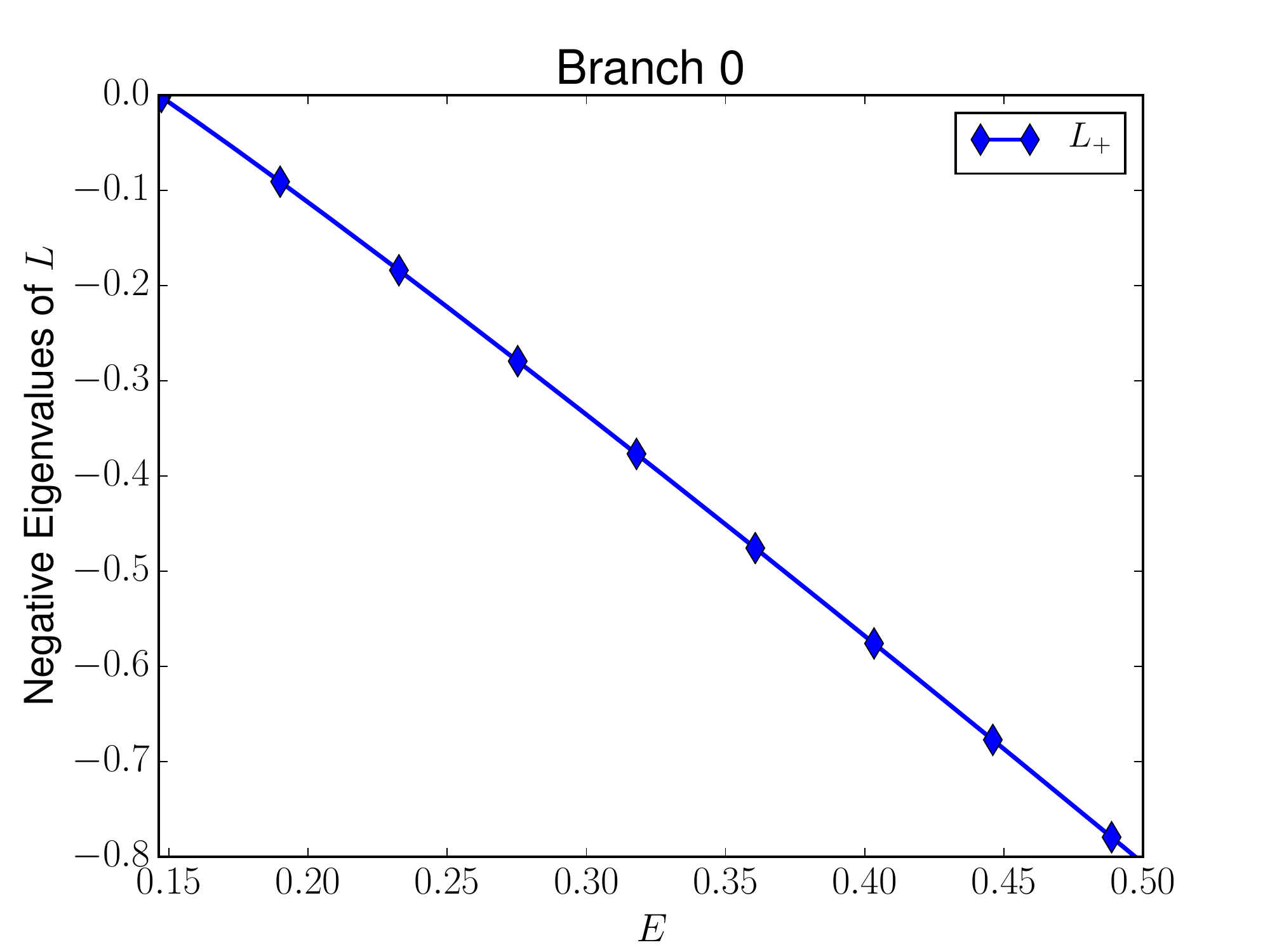}}
  \subfigure{\includegraphics[width=6.25cm]{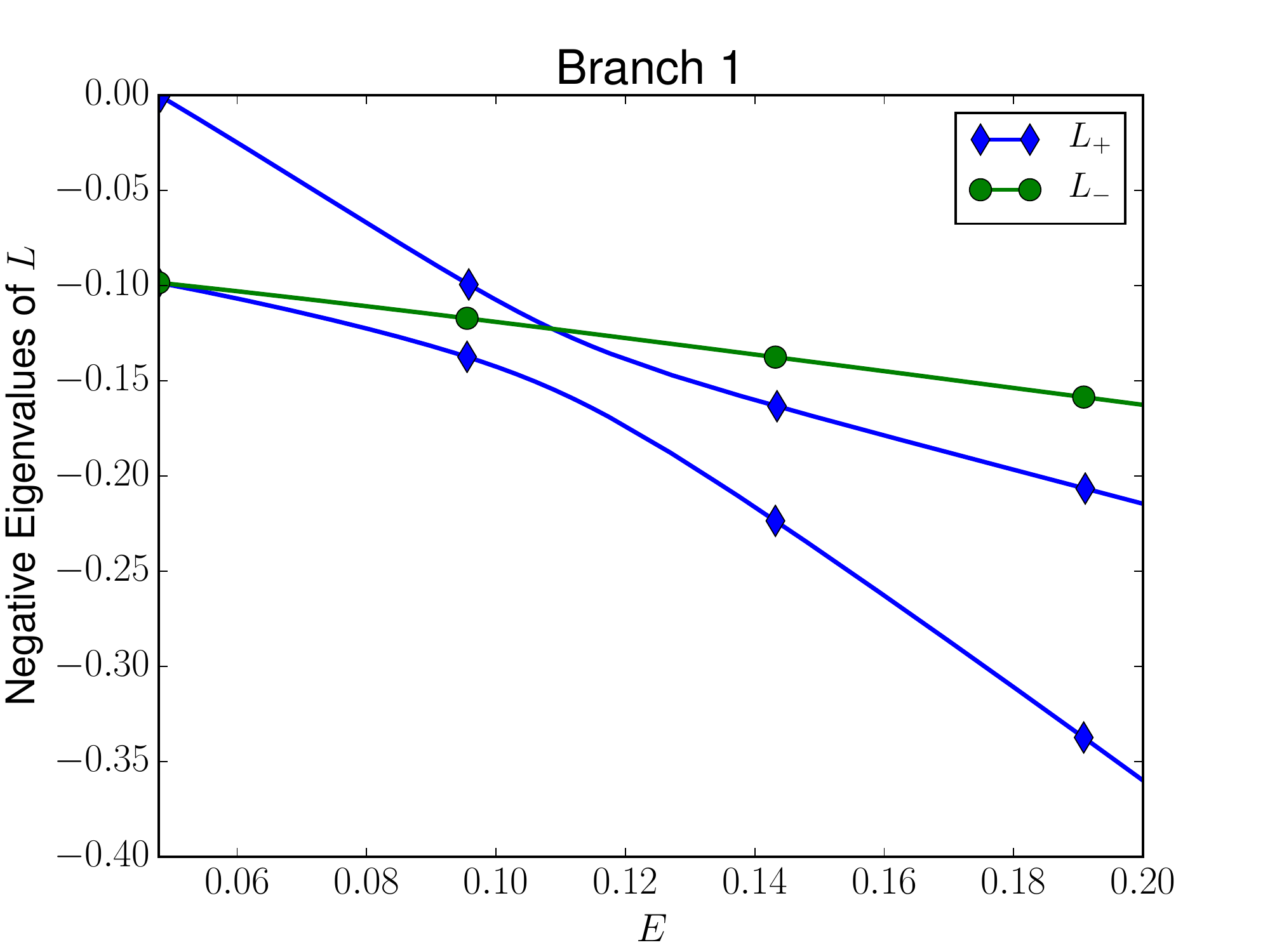}}

  \subfigure{\includegraphics[width=6.25cm]{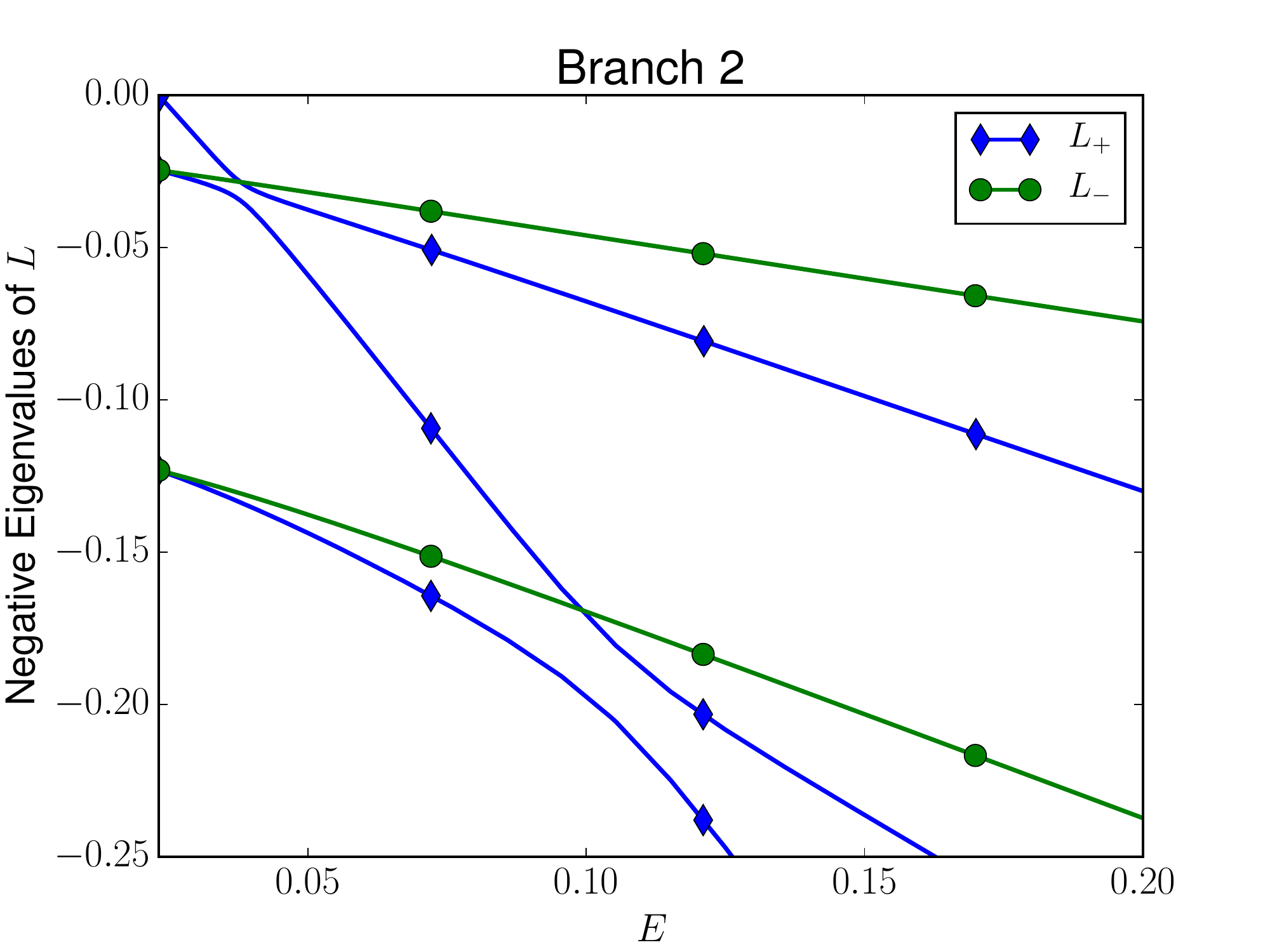}}
  \subfigure{\includegraphics[width=6.25cm]{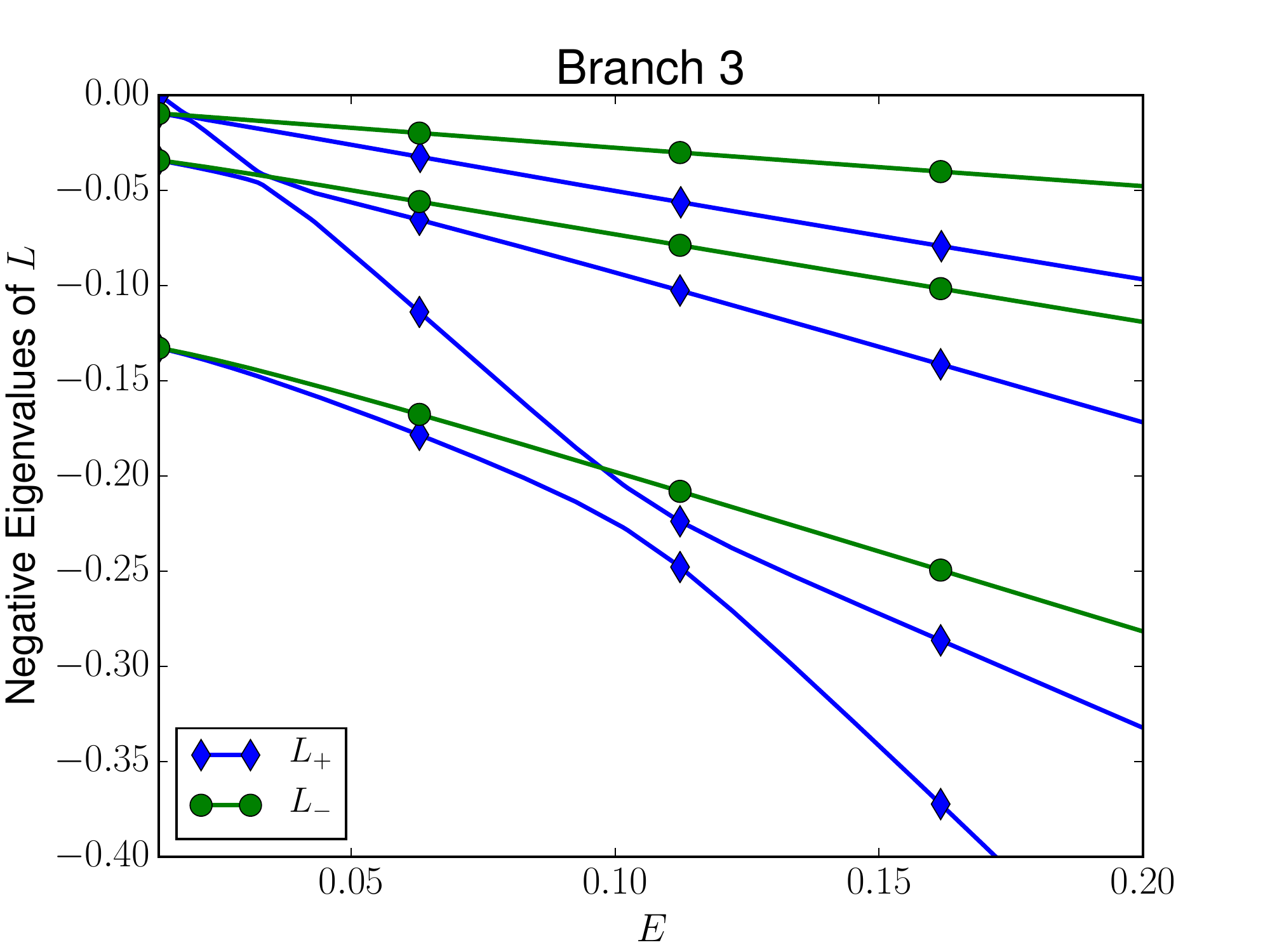}}

  \caption{Numerically computed negative spectrum of $L_\pm$ for the
    ground state and the first three excited states. Notice the
    crossings of the spectral lines in amongst the excited states.
    These were computed on \eqref{e:nonumesh} with $\rmax=100$ and
    $n=2000$.}
  \label{f:spec}
\end{figure}

\begin{remark}
  Our key assumption is that the kernel of $L_+$ remains trivial along
  the branches we have constructed.  This held in our numerical
  computations.
\end{remark}

Since this is inconclusive, we instead discretize $JL$ directly, and
examine its spectrum.  For the first few states, this is plotted in
Figure \ref{f:specJL} for $E=1$ solutions.  While the ground state has
no linearly unstable states, each of the excited states has some
number of linearly unstable modes, through the appearance of the
quartets of point spectra.  These were computed using the mesh
\eqref{e:nonumesh}, with $\rmax=100$ and $n=8000$.  While we were able
to use the automatically obtained mesh for computing the spectrum of
just $L_\pm$, this resulted in spurious purely real eigenvalues which
converged to the origin under mesh refinement.  Indeed, in the case of
the third branch, while not entirely visible, there is a pair of real
eigenvalues with magnitude $\bigo(10^{-4})$.  We believe these will
tend to zero under further mesh refinement, which we were unable to
do.

\begin{figure}
  \subfigure{\includegraphics[width=6.25cm]{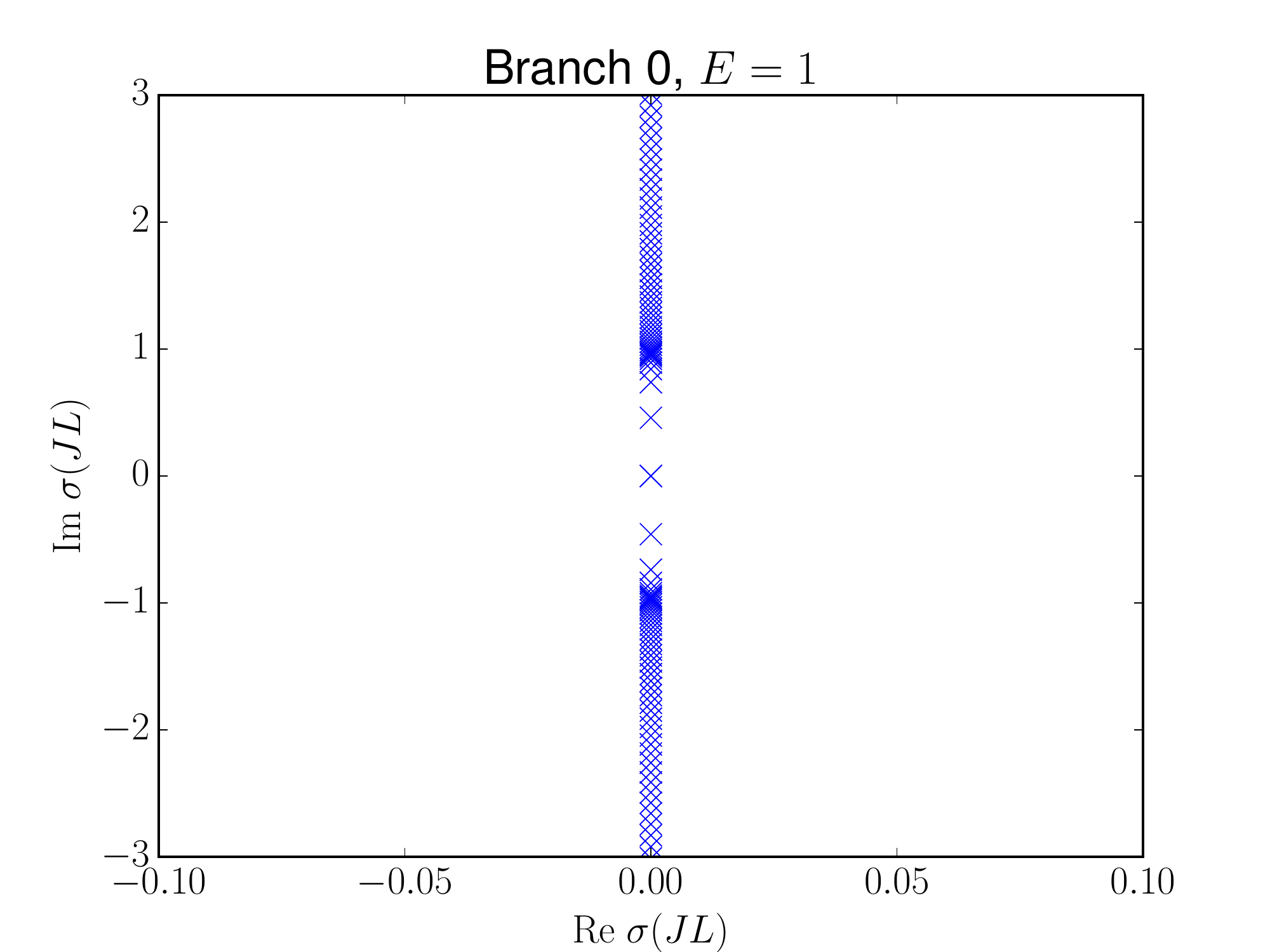}}
  \subfigure{\includegraphics[width=6.25cm]{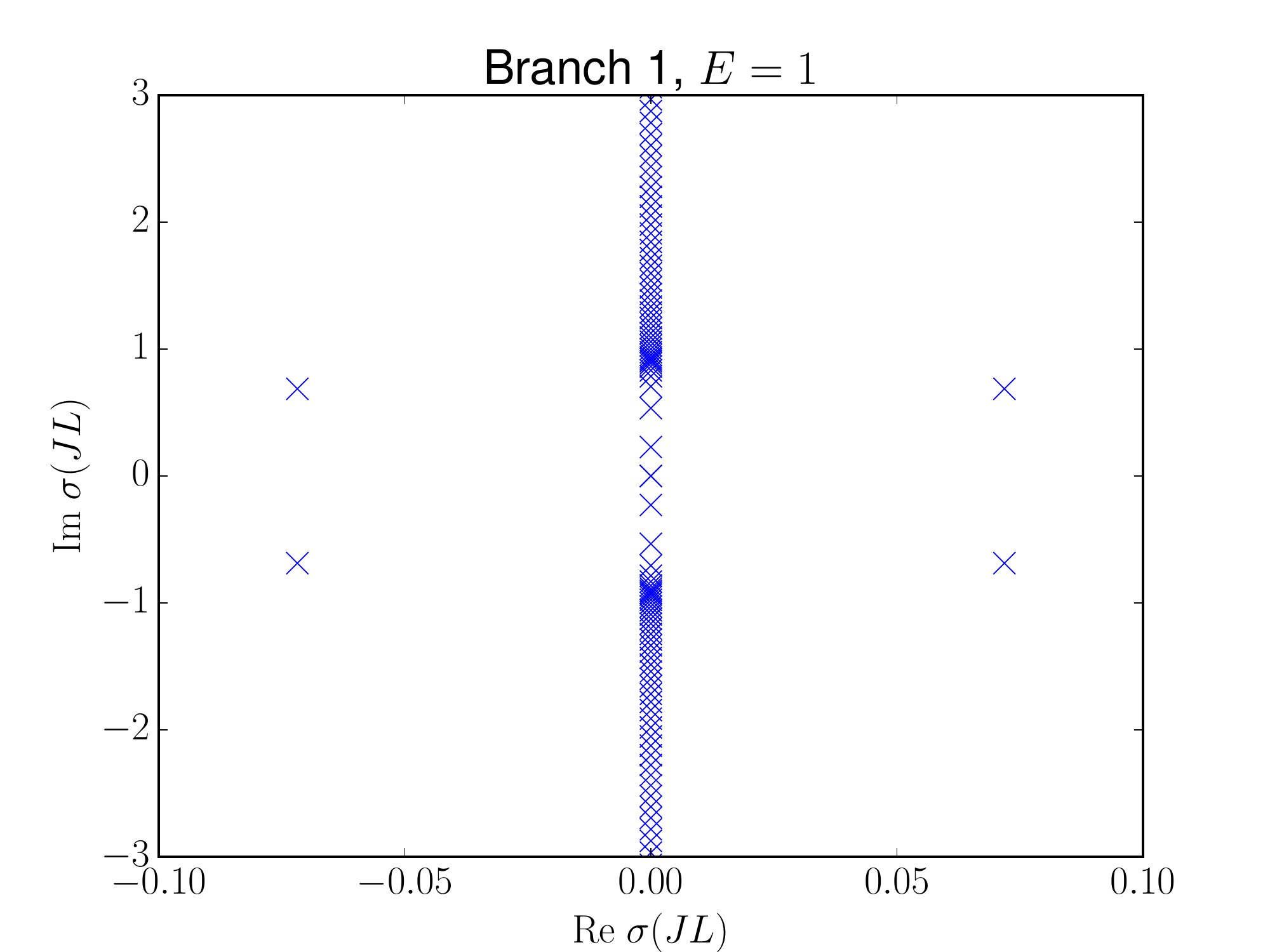}}

  \subfigure{\includegraphics[width=6.25cm]{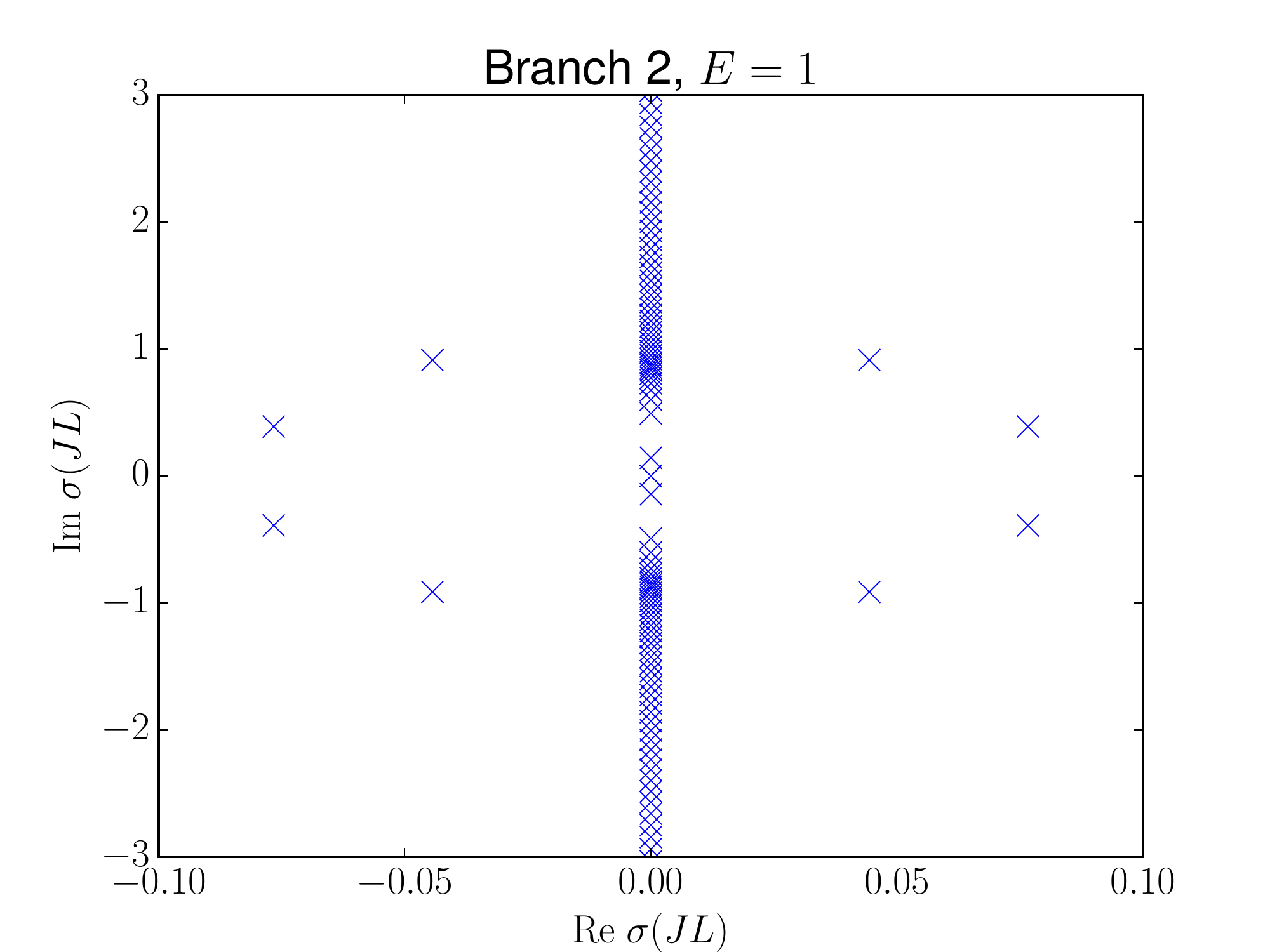}}
  \subfigure{\includegraphics[width=6.25cm]{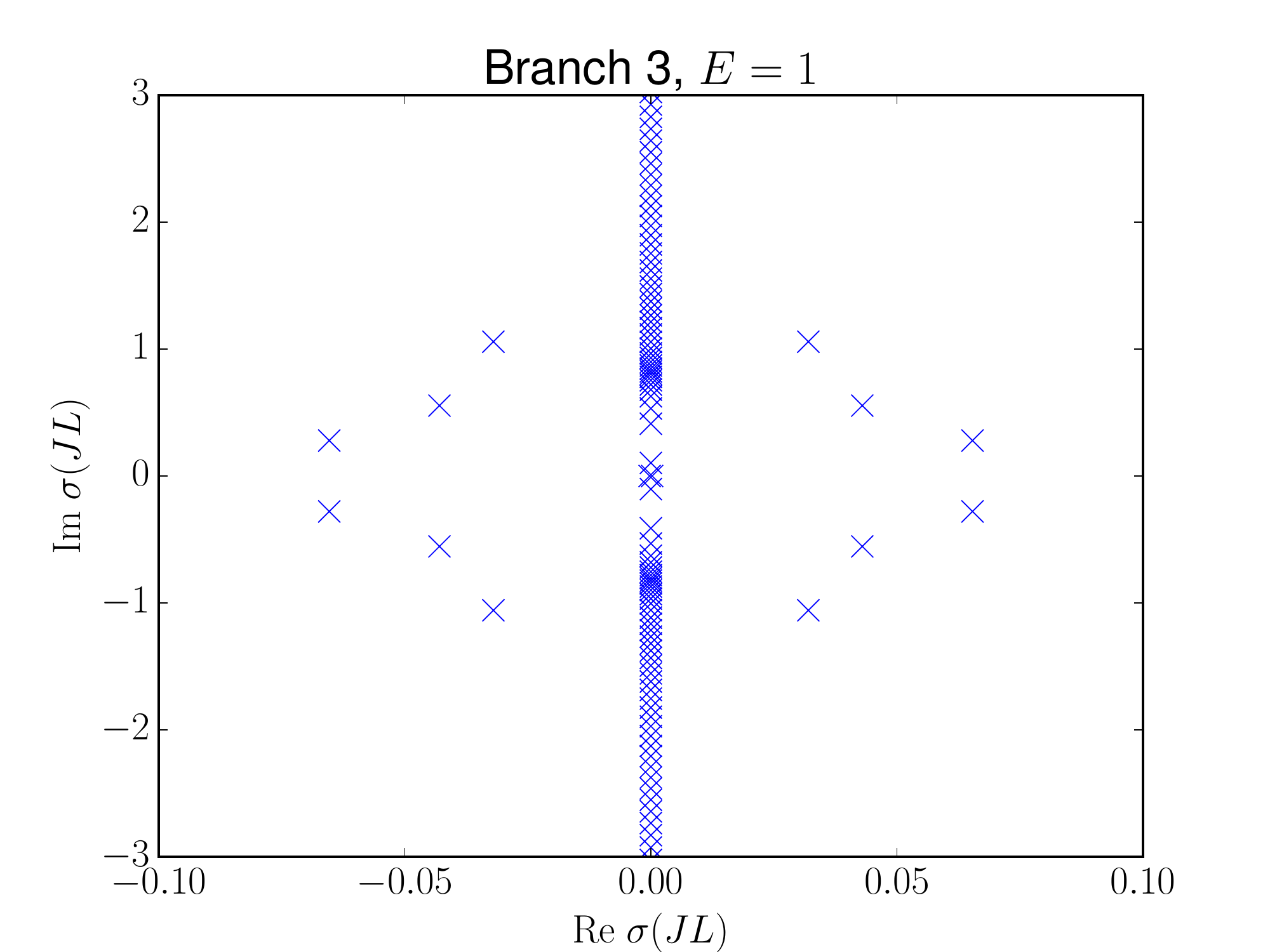}}
  \caption{Numerically computed spectrum for $JL$ for the ground state
    and the first three excited states.  Note the appearance of the
    quartets for the excited states implying linear instability.
    These were computed on \eqref{e:nonumesh} with $\rmax=100$ and
    $n=8000$.}
  \label{f:specJL}
\end{figure}

\subsubsection{Time-Dependent Simulations}

To assess the stability of the nonlinear bound states, we resort to
direct numerical simulation of
\begin{equation}
  \label{e:time_dep}
  i\partial_t \phi = - \Delta \phi + V(|x|) \phi - (-\Delta)^{-1}(\abs{\phi}^2) \phi
\end{equation}
using perturbations of the solutions we computed in the previous
section as initial conditions.  Indeed, our data is of the form
\begin{equation}
  \label{e:ic1}
  \phi_0 = u_j(r) + \epsilon \exp\left\{-4(r-10)^2\right\}
\end{equation}
with $\epsilon =10^{-4}$ and $u_j$ the $E=1$ solution of the $j$-th
branch.  We focus on the $E=1$ solutions, as these are highly
localized, decaying $\propto e^{-r}$.  Our results, pictured in Figure
\ref{f:timedep1} show that while the ground state appears to be
stable, the first excited state is unstable; this is consistent with
our spectral computations.  Throughout, we restrict to the radially
symmetric problem, and solve the initial boundary value problem
associated with \eqref{e:time_dep} on $(0, \rmax)$, with boundary
conditions
\begin{equation}
  \label{e:time_dep_bcs}
  \partial_r\phi(0,t) = 0, \quad\phi(\rmax,t) = 0.
\end{equation}
We made use of the mesh \eqref{e:nonumesh}.

\begin{figure}
  \subfigure{\includegraphics[width=6.25cm]{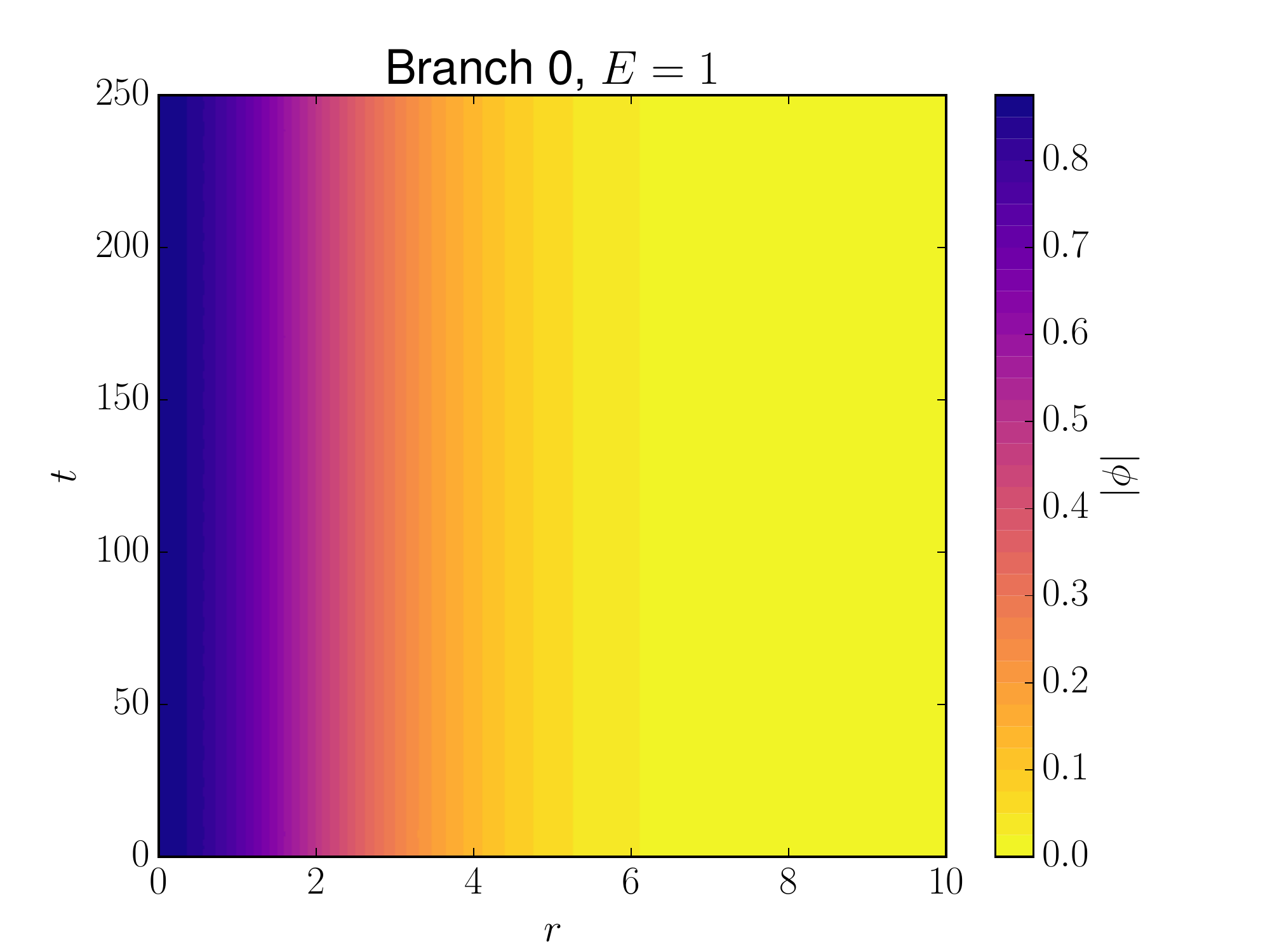}}
  \subfigure{\includegraphics[width=6.25cm]{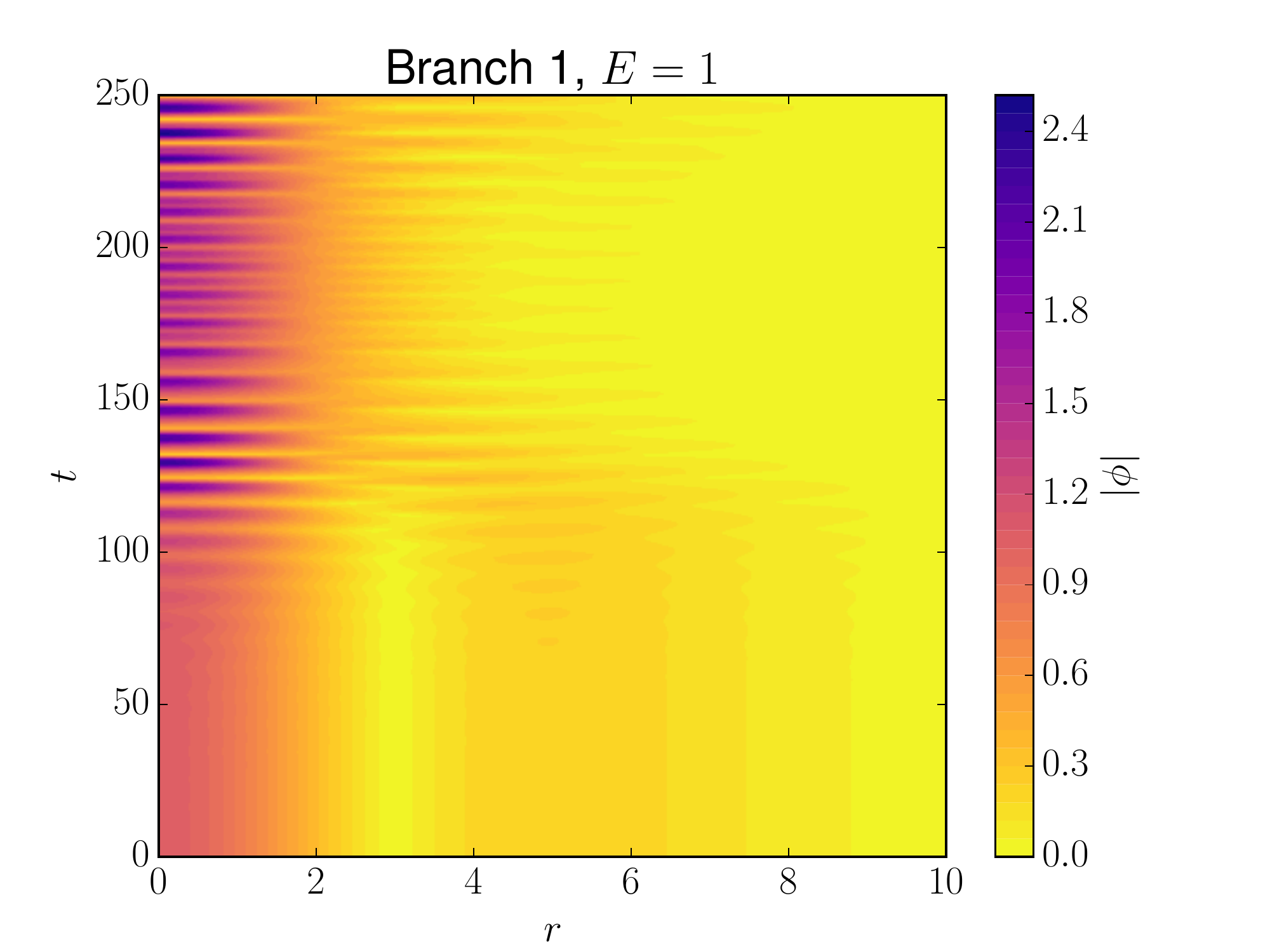}}
  \caption{Time dependent simulations of the ground state and first
    excited state solutions with $E=1$, with perturbed initial
    condition \eqref{e:ic1}.}
  \label{f:timedep1}
\end{figure}

Our algorithm is based on the Strang splitting method in
\cite{Lubich:2008hd}.  We solve \eqref{e:time_dep} through three
successive problems.  Given the solution at $t_n$, $\phi^{(n)}$,
\begin{align}
  \label{e:step1}
  \phi' &= \exp\set{\tfrac{i\Delta t}{2} (-\Delta)^{-1} |\phi^{{(n)}}|^2}\phi^{{(n)}}\\
  \label{e:step2}
  \phi'' & = \bracket{I  + \tfrac{i \Delta t}{2}(-\Delta+V) }^{-1}\bracket{I  - \tfrac{i \Delta t}{2}(-\Delta+V)}\phi'\\
  \label{e:step3}
  \phi^{(n+1)} &= \exp\set{\tfrac{i \Delta t}{2} (-\Delta)^{-1} |\phi'' |^2 }\phi''
\end{align}
Problem \eqref{e:step1} is accomplished by first solving
\begin{equation}
  -\Delta w= |\phi^{{(n)}}|^2, \quad w'(0) = 0, \quad \lim_{ r\to \infty} w(r) = 0
\end{equation}
on $(0, \rmax)$ with radial piecewise linear finite elements and the
Robin condition
\begin{equation*}
  w'(\rmax) + \tfrac{1}{\rmax}w(\rmax) = 0.
\end{equation*}
The FEM solution can be represented as
\begin{equation*} {\bf w} = I^{n,n+1}K_{\rob}^{-1} M
  I^{n+1,}|\boldsymbol{\phi}^{(n)}|^{2},
\end{equation*}
since $K_{\rob}$ is an $n+1\times n+1$ matrix, due to the Robin
condition, but our solution must satisfy the Dirichlet condition at
$\rmax$.  As before, $M$ is the mass matrix, and $K_{\rob}$ is the
stiffness matrix with Robin conditions.  The nonlinearity is
interpreted as an element-wise operation at the nodes.  Once we have
computed ${\bf w}$, we have
\begin{equation*}
  \boldsymbol{\phi}' = \exp\set{\tfrac{i\Delta t}{2} {\bf w}} \boldsymbol{\phi}^{(n)},
\end{equation*}
where, again, the operation is element-wise on the nodes.  Then,
\eqref{e:step2} is obtained from
\begin{equation}
  \boldsymbol{\phi}'' = \bracket{M + \tfrac{i \Delta t}{2}  K_{\dir} + U_{\mat}}^{-1}\bracket{M -\tfrac{i \Delta t}{2}  K_{\dir} + U_{\mat}} \boldsymbol{\phi}'.
\end{equation}
Finally \eqref{e:step3} is computed in the same way as
\eqref{e:step1}.  This method is efficient, as only sparse linear
algebra operations are required, and accurate.  The results shown in
Figure \ref{f:timedep1} were obtained on the nonuniform mesh
\eqref{e:nonumesh} with $\rmax = 4000$, $\Delta t = 0.00125$ and
$n=64000$.  They were stable to mesh refinement and other diagnostics.

To assess the accuracy of our simulations, we first examined the
conservation of the invariants, numerically approximated by
\begin{align}
  \mathcal{M}(\boldsymbol{\phi})  &= \boldsymbol{\phi}^T M
                                    \boldsymbol{\phi}  , \\
  \mathcal{H}(\boldsymbol{\phi})& = \boldsymbol{\phi}^T (K_{\dir} + U_{\mat}) \boldsymbol{\phi} -
                                  \tfrac{1}{2}(M|\boldsymbol{\phi}|^2)^T K_{\rob}^{-1}(M|\boldsymbol{\phi}|^2) .
\end{align}

The results for the simulations corresponding to Figure
\ref{f:timedep1} are shown in Figure \ref{f:conservation}.  While the
conservation of the ground state is excellent, there is a somewhat
larger discrepancy with the first excited state, although the relative
error over the lifetime of the simulation is still $\bigo(10^{-6})$.
To verify that there was no error, we performed convergence testing,
shown in Figure \ref{f:invconv}, indicating that the algorithm is
converging under mesh refinement, and further accuracy could be gained
with a reduction in $\Delta t$.

\begin{figure}
  \subfigure[Ground
  State]{\includegraphics[width=6.25cm]{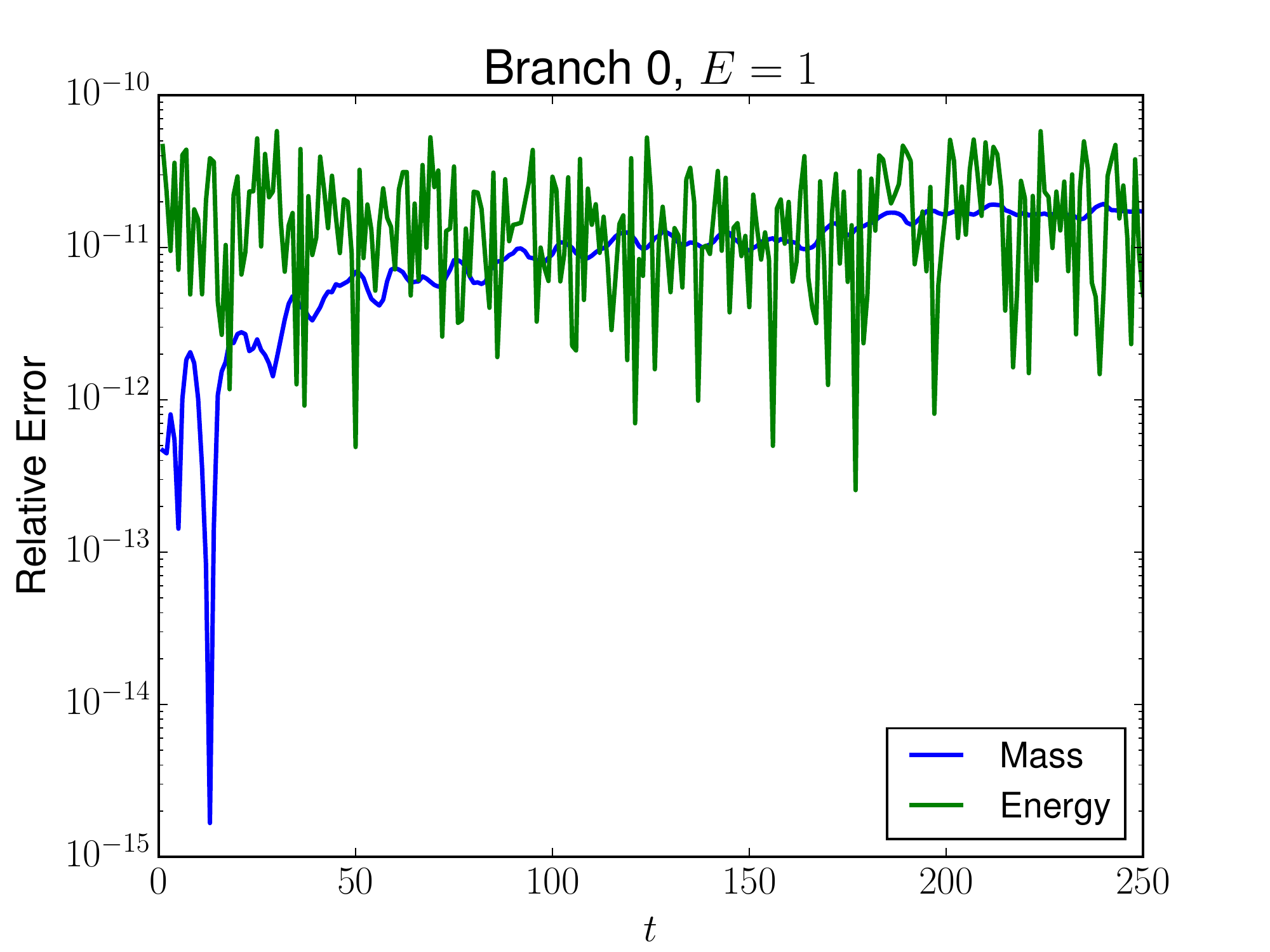}}
  \subfigure[First Excited
  State]{\includegraphics[width=6.25cm]{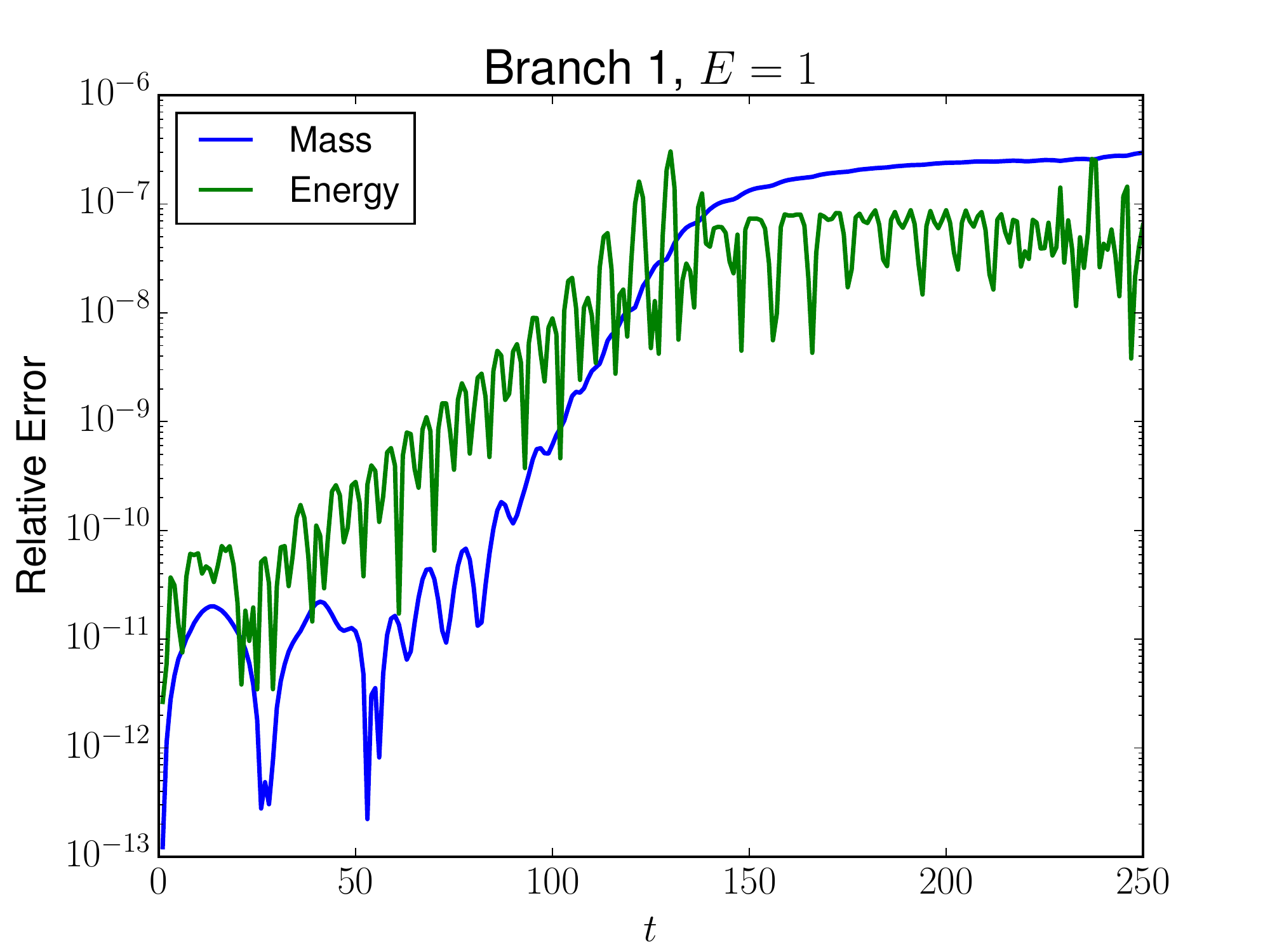}}
  \caption{Conservation of the numerical invariants under during the
    simulations.  These are representative of our results and
    correspond to Figure \ref{f:timedep1}.}
  \label{f:conservation}
\end{figure}

\begin{figure}
  \subfigure[]{\includegraphics[width=6.25cm]{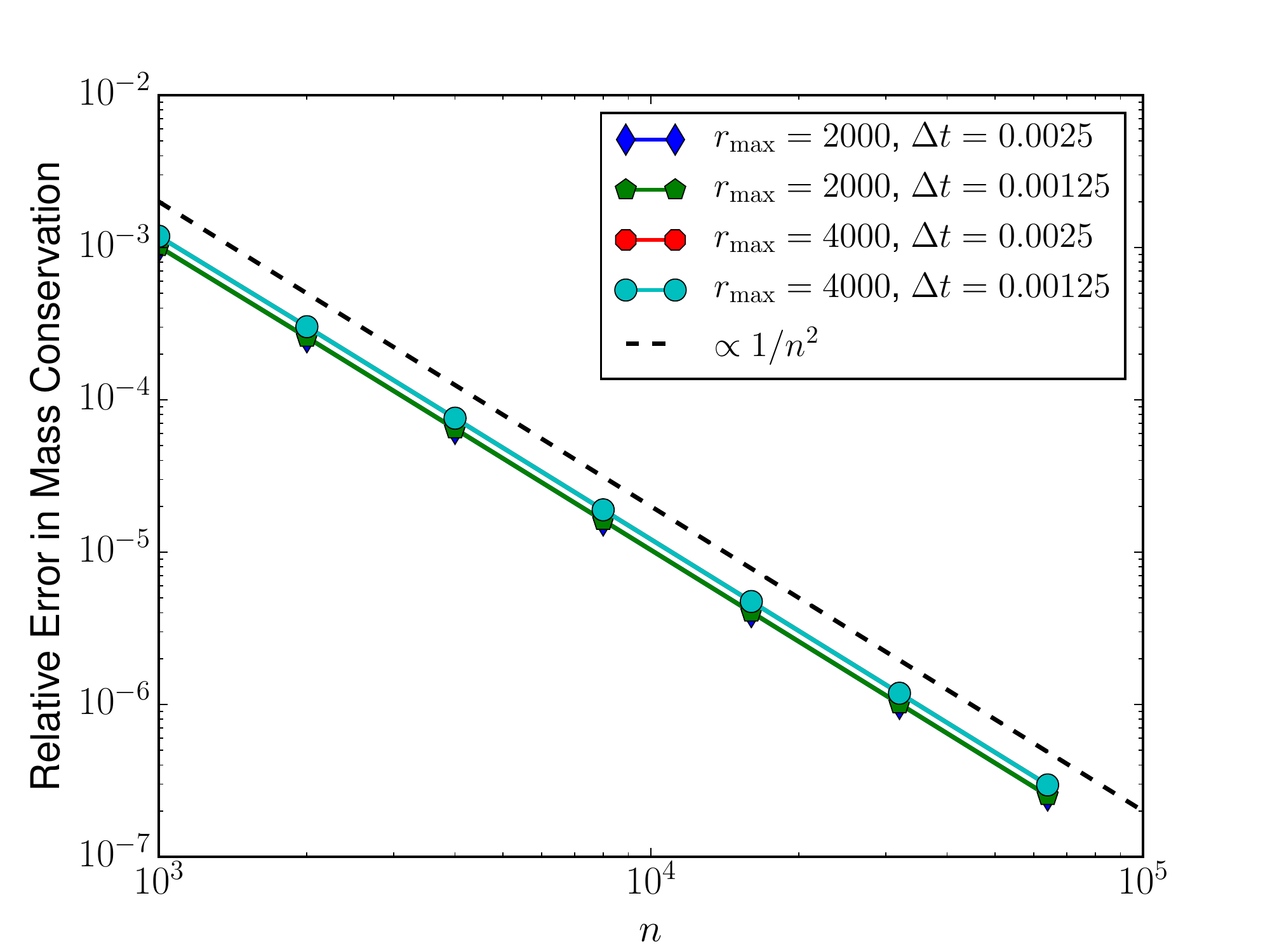}}
  \subfigure[]{\includegraphics[width=6.25cm]{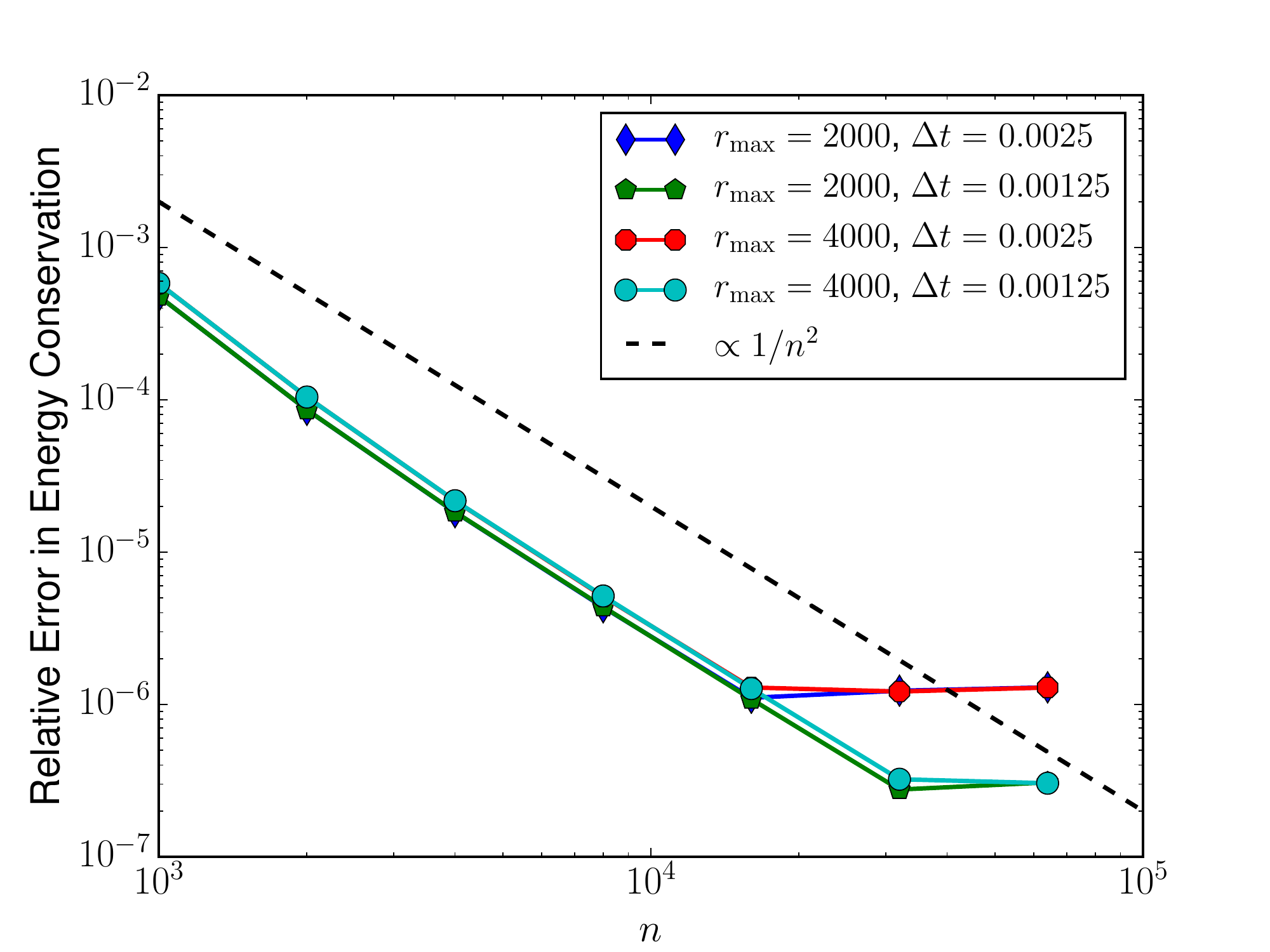}}
  \caption{Convergence of the invariants for the time dependent
    simulations of the perturbed first excited state.}
  \label{f:invconv}
\end{figure}

The only other comment we make on our methodology is that $\rmax$ must
be sufficiently large to allow for the homogeneous Dirichlet condition
at $\rmax$ in \eqref{e:time_dep_bcs}.  While we used $\rmax=100$ to
compute the nonlinear bound states, for the time-dependent problems we
took $\rmax = 2000$ and $\rmax=4000$.  We thus matched the computed
$u$ on $[0,100]$ to the far field asymptotics,
\begin{equation}
  \label{e:matched}
  u\sim K e^{-\sqrt{E} r} r^{-1 +\frac{1 + m(100)}{2 \sqrt{E}}},
\end{equation}
to generate an initial condition.

In the time dependent simulations, the solution was typically smaller
than $\bigo(10^{-8})$ at the boundary throughout the simulation when
$\rmax=2000$ and smaller than $\bigo(10^{-15})$ when $\rmax =4000$.
This, together with other convergence testing in time step, domain
size, and mesh spacing leads us to believe that the stability and
instability results we have observed are genuine and not numerical.
They are also consistent with simulations appearing in
\cite{harrison2002numerical}, where the authors examined
\eqref{SchPoi} in the setting $V=0$.  There, they found the ground
state to be stable and the first excited state to be unstable,
agreeing with our simulations at $E=1$.

\subsection{Transitions to Instability}

While for $E=1$, the $JL$ spectral computations and time dependent
simulations reveal linear instabilities of the excited state branches,
the stability question for general $E$ remains unresolved.  As Figure
\ref{f:transition1} shows, there may be some stable excited state
solutions.  In the figure, we have plotted the maximum of the real
part of the $JL$ spectrum for a series of branch 1 excited
states. These were computed using the nonuniform mesh with
$\rmax = 100$ and $n=2000$.  Examining the figure, there appears to be
a secondary bifurcation near $E=0.13$, where a linearly unstable
eigenvalue first appears.  Beneath that value, no such linear
instability is present.  With regard to bounds on unstable
eigenvalues, we calculate in Appendix \ref{a:bds} that
\begin{equation}
  \sup \abs{\Re \sigma(JL)}\leq C \|u_{E}\|_{L^2}^2 \sqrt{\frac{1}{3}E +
    \frac{1}{3} \|V\|_{L^\infty} + \frac{2}{3}\|x\nabla V\|_{L^\infty}},
\end{equation}
ensuring that unstable eigenvalues must vanish in the zero mass limit.


\begin{figure}
  \includegraphics[width=6.25cm]{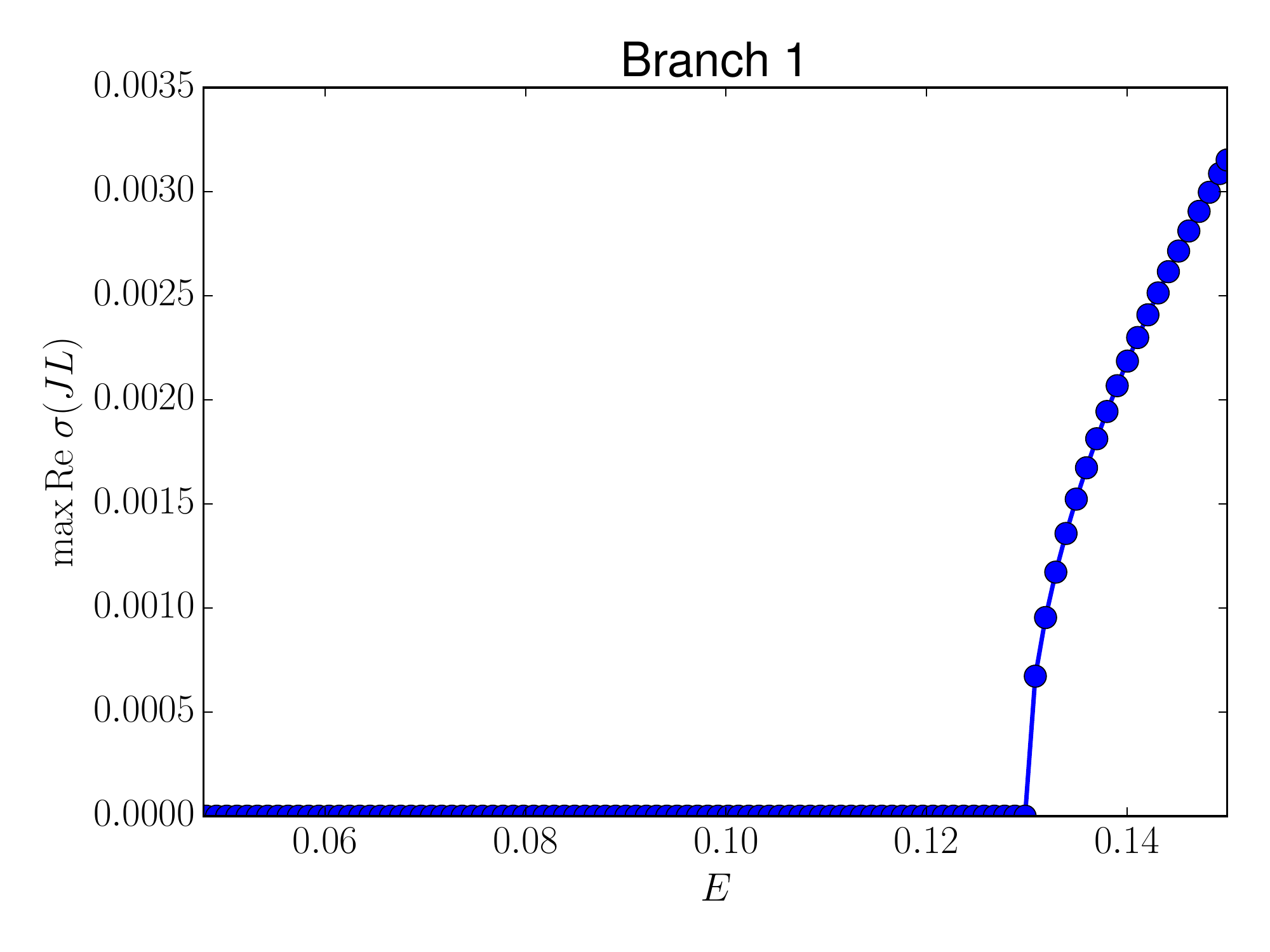}
  \caption{Numerically computed spectrum for $JL$ for the first
    excited state over a range of $E$ values, from the mass zero
    bifurcation value up to 0.15.  Note the secondary bifurcation near
    $E=0.13$, where unstable eigenvalues appear.  These correspond to
    quartets. when plotted in the complex plane.  Computed on
    \eqref{e:nonumesh} with $\rmax=100$ and $n=2000$.}
  \label{f:transition1}
\end{figure}

\section{Discussion}
\label{s:disc}

We have analytically and numerically explored radial nonlinear bound
state solutions of the Schr\"odinger-Poisson equation with an
attractive Coulomb like potential.  These states were shown to branch
off of the discrete modes of the associated linear problem.  Subject
to a spectral assumption, these can be continued to have arbitrarily
large mass and large $E$ parameter.

Our numerical methods for computing the solutions, first computing the
linear modes, and then performing continuation in an artificial
parameter, was robust.  Subsequent time dependent simulations, using a
FEM discretization and a splitting scheme, also proved themselves to
be robust, showing excellent conservation of the invariants.

While our work implies the stability of the ground state at all values
of $E$, we are unable to make any broad conclusions about the excited
states.  Our spectral and time dependent computations for the $E=1$
solutions imply they are linearly unstable.  At the same time, in our
examination of the $JL$ spectrum near the zero mass limit for branch
1, we see a subsequent bifurcation with the emergence of linearly
unstable modes.  Further investigation is necessary to explore the
other states, and a spectral approach, such as the one used in
\cite{harrison2002numerical} for $V=0$, may be of use.  We also note
the recent work \cite{cuccagna2015orbital}, in which it is shown that
the spectrally stable excited states can still be nonlinearly unstable
through radiation damping and the Fermi Golden Rule, which require
dispersive decay that is not understood for long range, Coulomb-style
potentials.

In most applications, excited states can be both linearly unstable, as
observed here for large enough $E$, as well as orbitally unstable due
to radiation damping.  For long range potentials of the case studied
here, it is unknown how to quantify radiation effects.  See, for
instance, the work \cite{tsai2002relaxation} for a discussion of
resonant interactions and \cite{KKP,KKSW,pelinovsky2012normal} in the
setting of bifurcation theory for the excited state of a localized
double well potential.

The promise of stable excited states in our setting is relevant given
the observations in \cite{Bray}, that excess energy contained in a
dwarf spheroidal galaxy corresponds to larger essential support of the
dark matter field. By essential support here, we mean the volume of
space on which the solution has non-trivial mass.  This is
contradictory to the nature of scaling of the ground state, which
contracts its effective support as mass is increased in our model.
Our numerics imply that for sufficiently large $E$, excited state
branches are unstable.  More refined numerical analysis and theory is
required to determine if the conjecture of \cite{Bray} about stability
of excited state branches holds neared the bifurcation points.  This
merits future analytical and computational study.




\appendix

\section{Bounds on Unstable Eigenvalues}
\label{a:bds}

To bound the positive real part of the eigenvalues of $JL$, we follow
the results from \cite{tod2001ground} and \cite[Appendix
A]{harrison2002numerical}.  First, we express the eigenvlaue problem
as
\begin{align}
  L_-v &= \lambda u,\\
  L_+ u & = - \lambda v,
\end{align}
with
\begin{equation}
  L_+ = L_- - 2 |x|^{-1} \ast(u_E \bullet) u_E.
\end{equation}
Therefore
\begin{equation}
  \int \bar u L_- v - \int v L_+ \bar u = \int \lambda |u|^2 + \int \bar \lambda \bar{|v|^2}
\end{equation}
Writing $\lambda = \sigma + i \tau$, and using the self adjointness of
$L_-$,
\begin{equation}
  \sigma(\|u\|_{L^2}^2 + \|v\|_{L^2}^2) + i \tau (\|u\|_{L^2}^2 -
  \|v\|_{L^2}^2) = \int 2 |x|^{-1}(u_E \bar u) u_E v
\end{equation}
Therefore, taking real parts and then absolute values,
\begin{equation}
  \begin{split}
    |\sigma|(\|u\|_{L^2}^2 + \|v\|_{L^2}^2)&=2\left|\Re \int
      |x|^{-1}(u_E \bar u) u_E v\right|\\
    & \leq 2 \abs{\int |x|^{-1}(u_E \bar u) u_E v }
  \end{split}
\end{equation}
By Hardy-Littlewood-Sobolev,
\begin{equation}
  \abs{\int |x|^{-1}(u_E \bar u) u_E v} \leq C_{\rm HLS} \|u_E u\|_{L^{6/5}}\|u_E v\|_{L^{6/5}}
\end{equation}
By H\"older,
\begin{equation}
  \|u_E u\|_{L^{6/5}}\leq \|u_E\|_{L^3} \|u\|_{L^2} 
\end{equation}
Therefore,
\begin{equation}
  |\sigma|(\|u\|_{L^2}^2 + \|v\|_{L^2}^2)\leq 2 C_{\rm HLS}
  \|u_E\|_{L^3}^2 \|u\|_{L^2} \|v\|_{L^2}\leq C_{\rm HLS}\|u_E\|_{L^3}^2 (\|u\|_{L^2}^2 + \|v\|_{L^2}^2)
\end{equation}
Which gives our first bound:
\begin{equation}
  |\sigma| \leq C_{\rm HLS}\|u_E\|_{L^3}^2
\end{equation}
We would like to have that as $E$ tends to the bifurcation value,
$\|u_E\|_{L^3}$ tends to zero.  By H\"older again, and
Gagliardo-Nirenberg,
\begin{equation}
  \|u_E\|_{L^3}\leq \sqrt{\|u_E\|_{L^6} \|u_E\|_{L^2}} \leq \sqrt{C_{\rm
      GN}\|\nabla u_E\|_{L^2}   \|u_E\|_{L^2} }
\end{equation}
Therefore,
\begin{equation}
  |\sigma| \leq C_{\rm HLS}C_{\rm GN} \|\nabla u_E\|_{L^2}  \|u_E\|_{L^2}
\end{equation}

A bit more refinement can be done.  To get further control in terms of
$L^2$ norm alone, one can modify the energy-momentum tensor tensor
techniques as applied in \cite{tod2001ground} and \cite[Appendix
A]{harrison2002numerical}, to prove
\begin{equation}
  \label{KEbd}
  \int | \nabla u_E|^2 dx =  \frac13 E \int |u_E|^2 dx - \frac13 \int V |u_E|^2 + \frac23 \int x\cdot \nabla V |u_E|^2 dx.
\end{equation}
As a result, we have
\begin{equation}
  |\sigma| \leq C_{\rm HLS}C_{\rm GN}  \|u_E\|_{L^2}^2
  \sqrt{\frac{1}{3}E + \frac{1}{3}\|V\|_{L^\infty} +
    \frac{2}{3}\|x\nabla V\|_{L^\infty} }
\end{equation}
Thus we obtain a bound entirely in terms of $E$ and $\|u_E\|_{L^2}$.

%

To prove estimate \eqref{KEbd}, recall the energy-momentum tensor as
applied in \cite{tod2001ground},
\begin{align*}
  T_{ij} &= \partial_i u_E \partial_j (\bar{u}_E) + \partial_j u_E \partial_i (\bar{u}_E) + \partial_i (|x|^{-1} *  |u_E|^2)  \partial_j ( |x|^{-1} *  |u_E|^2) \\
         & - \delta_{ij} \left(  \sum_{k=1}^3  (  \partial_k u_E \partial_k (\bar{u}_E)  + \frac12 \partial_k (|x|^{-1} *  |u_E|^2)  \partial_k ( |x|^{-1} *  |u_E|^2)  )   \right. \\
         & \left.  +  \int V |u_E|^2 dx -  ( |x|^{-1} *  |u_E|^2) |u_E|^2 + E |u_E|^2 \right).  
\end{align*}
This is identical up to the stress-energy tensor from \cite[Appendix
A]{harrison2002numerical} modulo terms with $V$.  We claim that for
$i=1,2,3$,
\[
  T_{ij,j} = \sum_{j=1}^3 \partial_j (T_{ij}) = -(\partial_i V).
  |u_E|^2
\]
To see this, observe that
\begin{align*}
  & \frac12 \partial_1 ( | \partial_1 u_E|^2 - |\partial_2 u_E|^2 - |\partial_3 u_E|^2)+ \partial_2 (\partial_1 u_E \partial_2 \bar{u}_E ) + \partial_3 ( \partial_1 u_E \partial_3 \bar{u}_E)  \\
  & \hspace{2cm} = (\Delta \bar{u}_E )  \partial_1 u_E .
\end{align*}
For convenience, let us set $ \phi = (x^{-1} \ast |u_E|^2)$, the we
have
\begin{align*}
  T_{1j,j} & =(\Delta \bar{u}_E )  \partial_1 u_E  + (\Delta u_E )\partial_1 \bar{u}_E + ( \partial_1 \phi) \Delta \phi -  \partial_1 ( V |u_E|^2 - \phi |u_E|^2 + E |u_E|^2  ) \\
           &  =(V \bar{u}_E - \phi \bar{u}_E + E \bar{u}_E )  \partial_1 u_E  + (V u_E - \phi u_E + E u_E )\partial_1 \bar{u}_E - ( \partial_1 \phi) |u_E|^2 \\
           & -  \partial_1 ( V |u_E|^2-\phi |u_E|^2 + E |u_E|^2  ) \\
           & =- (\partial_1 V) |u_E|^2,
\end{align*} 
where we have used that $-\Delta \phi = |u_E|^2$.  A similar
calculation for $i = 2,3$.  Then,
\begin{equation*}
  \sum_{i=1}^3  \sum_{j= 1}^3 \partial_i ( T_{ij} x_j)  = \sum_{i =1}^3 T_{ii} - ( x\cdot \nabla V) |u_E|^2.
\end{equation*}
This implies
\begin{equation*}
  \int  \sum_{i =1}^3 T_{ii} dx - \int (x \cdot \nabla V) |u_E|^2 dx = 0,
\end{equation*}
and hence
\begin{align*}
  & 0 = \int \big(  - | \nabla u_E |^2 - \frac12 | \nabla  \phi |^2 + 3  \phi |u_E|^2 - 3 V |u_E|^2  \\
  & \hspace{2cm}   - 3 E |u_E|^2 - x\cdot \nabla V |u_E|^2 \big) dx  .\notag
\end{align*}
Recognizing that
\[
  \int | \nabla \phi |^2 dx = - \int \Delta \phi \phi = - \int \phi
  |u_E|^2 dx
\]
and using that
\[
  -E \int |u_E|^2 dx = \int | \nabla u_E|^2 dx + \int V |u_E|^2 dx -
  \int \phi |u_E|^2 dx,
\]
we arrive at \eqref{KEbd}.


\bibliographystyle{plain}

\bibliography{schpoisson}

\begin{thebibliography}{10}

\bibitem{anantharaman2009existence}
Arnaud Anantharaman and Eric Canc{\`e}s.
\newblock Existence of minimizers for {K}ohn--{S}ham models in quantum
  chemistry.
\newblock {\em Annales de l'Institut Henri Poincare (C) Non Linear Analysis},
  26(6):2425--2455, 2009.

\bibitem{benguria1981thomas}
R.~Benguria, H.~Br{\'e}zis, and E.~H. Lieb.
\newblock {The Thomas-Fermi-von Weizs{\"a}cker theory of atoms and molecules}.
\newblock {\em Communications in Mathematical Physics}, 79(2):167--180, 1981.

\bibitem{Bray}
H.L. Bray.
\newblock On dark matter, spiral galaxies, and the axioms of general
  relativity.
\newblock {\em Geometric Analysis, Mathematical Relativity, and Nonlinear
  Partial Differential Equations}, 599:1--64, 2010.

\bibitem{BG}
H.L. Bray and A.S. Goetz.
\newblock {Wave Dark Matter and the Tully-Fisher Relation}.
\newblock arXiv:1409.7347, 2014.

\bibitem{BP}
H.L Bray and A.R. Parry.
\newblock Modeling wave dark matter in dwarf spheroidal galaxies.
\newblock In {\em Journal of Physics: Conference Series}, volume 615, page
  012001. IOP Publishing, 2015.

\bibitem{coddington1955theory}
Earl~A Coddington and Norman Levinson.
\newblock {\em Theory of ordinary differential equations}.
\newblock Tata McGraw-Hill Education, 1955.

\bibitem{cuccagna2015orbital}
Scipio Cuccagna and Masaya Maeda.
\newblock On orbital instability of spectrally stable vortices of the {NLS} in
  the plane.
\newblock {\em arXiv preprint arXiv:1508.03146}, 2015.

\bibitem{Grillakis}
M.~Grillakis.
\newblock {Linearized instability for nonlinear schr{\"o}dinger and
  Klein-Gordon equations}.
\newblock {\em Communications on pure and applied mathematics}, 41(6):747--774,
  1988.

\bibitem{Grillakis:1990p116}
M.~Grillakis.
\newblock {Analysis of the linearization around a critical point of an infinite
  dimensional Hamiltonian system}.
\newblock {\em Communications on Pure and Applied Mathematics}, 43(3):299--333,
  1990.

\bibitem{GSS}
M.~Grillakis, J.~Shatah, and W.~Strauss.
\newblock Stability theory of solitary waves in the presence of symmetry, i.
\newblock {\em Journal of Functional Analysis}, 74(1):160--197, 1987.

\bibitem{gustafson2011mathematical}
S.J. Gustafson and I.M. Sigal.
\newblock {\em Mathematical concepts of quantum mechanics}.
\newblock Springer Science \& Business Media, 2011.

\bibitem{harrison2002numerical}
Richard Harrison, Irene Moroz, and KP~Tod.
\newblock A numerical study of the {S}chr{\"o}dinger--{N}ewton equations.
\newblock {\em Nonlinearity}, 16(1):101, 2002.

\bibitem{scipy}
Eric Jones, Travis Oliphant, Pearu Peterson, et~al.
\newblock {SciPy}: Open source scientific tools for {Python}, 2001--.
\newblock [Online; accessed 2016-04-20].

\bibitem{KKP}
E.~Kirr, P.G. Kevrekidis, and D.E. Pelinovsky.
\newblock {Symmetry-breaking bifurcation in the nonlinear Schr{\"o}dinger
  equation with symmetric potentials}.
\newblock {\em Communications in mathematical physics}, 308(3):795--844, 2011.

\bibitem{KKSW}
E.W. Kirr, P.G. Kevrekidis, E.~Shlizerman, and M.I. Weinstein.
\newblock {Symmetry-breaking bifurcation in nonlinear
  Schr{\"o}dinger/Gross-Pitaevskii equations}.
\newblock {\em SIAM Journal on Mathematical Analysis}, 40(2):566--604, 2008.

\bibitem{KM}
R.~Koll{\'a}r and P.D. Miller.
\newblock {Graphical Krein Signature Theory and Evans--Krein Functions}.
\newblock {\em SIAM Review}, 56(1):73--123, 2014.

\bibitem{Lenzmann}
E.~Lenzmann.
\newblock {Uniqueness of ground states for pseudorelativistic Hartree
  equations}.
\newblock {\em Analysis \& PDE}, 2(1):1--27, 2009.

\bibitem{LS1977}
E.H. Lieb and B.~Simon.
\newblock {The Hartree-Fock theory for Coulomb systems}.
\newblock {\em Communications in Mathematical Physics}, 53(3):185--194, 1977.

\bibitem{lions1980choquard}
PL~Lions.
\newblock The choquard equation and related questions.
\newblock {\em Nonlinear Analysis: Theory, Methods \& Applications},
  4(6):1063--1072, 1980.

\bibitem{Lubich:2008hd}
C.~Lubich.
\newblock {On splitting methods for Schr\"odinger-Poisson and cubic nonlinear
  Schr\"odinger equations}.
\newblock {\em Mathematics Of Computation}, 77(264):2141--2153, 2008.

\bibitem{Olson:2014ix}
D.~Olson, S.~Shukla, G.~Simpson, and D.~Spirn.
\newblock {Petviashvilli{\textquoteright}s Method for the Dirichlet Problem}.
\newblock {\em Journal of Scientific Computing}, pages 1--25, 2014.

\bibitem{pelinovsky2012normal}
DE~Pelinovsky and TV~Phan.
\newblock Normal form for the symmetry-breaking bifurcation in the nonlinear
  schr{\"o}dinger equation.
\newblock {\em Journal of Differential Equations}, 253(10):2796--2824, 2012.

\bibitem{RS4}
M.~Reed and B.~Simon.
\newblock {\em Analysis of Operators, Vol. IV of Methods of Modern Mathematical
  Physics}.
\newblock New York, Academic Press, 1978.

\bibitem{sakaguchi2011suppression2}
Hidetsugu Sakaguchi and Boris~A Malomed.
\newblock Suppression of quantum collapse in an anisotropic gas of dipolar
  bosons.
\newblock {\em Physical Review A}, 84(3):033616, 2011.

\bibitem{sakaguchi2011suppression1}
Hidetsugu Sakaguchi and Boris~A Malomed.
\newblock Suppression of the quantum-mechanical collapse by repulsive
  interactions in a quantum gas.
\newblock {\em Physical Review A}, 83(1):013607, 2011.

\bibitem{Shampine:2006vr}
L.F. Shampine, P.H. Muir, and H.~Xu.
\newblock {A User-Friendly Fortran BVP Solver}.
\newblock {\em JNAIAM}, 1(2):201--217, 2006.

\bibitem{tod2001ground}
KP~Tod.
\newblock The ground state energy of the {S}chr{\"o}dinger--{N}ewton equation.
\newblock {\em Physics Letters A}, 280(4):173--176, 2001.

\bibitem{tsai2002relaxation}
Tai-Peng Tsai and Horng-Tzer Yau.
\newblock Relaxation of excited states in nonlinear schr{\"o}dinger equations.
\newblock {\em International Mathematics Research Notices},
  2002(31):1629--1673, 2002.

\bibitem{Weinstein85}
M.I. Weinstein.
\newblock Modulational stability of ground states of nonlinear schr{\"o}dinger
  equations.
\newblock {\em SIAM journal on mathematical analysis}, 16(3):472--491, 1985.

\bibitem{Weinstein}
M.I. Weinstein.
\newblock Lyapunov stability of ground states of nonlinear dispersive evolution
  equations.
\newblock {\em Communications on Pure and Applied Mathematics}, 39(1):51--67,
  1986.

\end{thebibliography}

\end{document}